\documentclass{amsart}
\usepackage{amsmath}
\usepackage{verbatim}
\usepackage{amsfonts}
\usepackage{graphicx}
\usepackage{caption}
\usepackage{color}
\usepackage{subcaption}
\usepackage[english]{babel}
\newtheorem{theorem}{Theorem}[section]
\newtheorem{lemma}[theorem]{Lemma}
\newtheorem{corollary}[theorem]{Corollary}
\newtheorem{proposition}[theorem]{Proposition}

\theoremstyle{definition}
\newtheorem{definition}[theorem]{Definition}

\theoremstyle{remark}
\newtheorem{remark}[theorem]{Remark}
\numberwithin{equation}{section}

\newcommand{\eps}{\epsilon}

\newcommand{\Var}{\text{Var}}

\newcommand{\PF}{\mathcal{L}}
\newcommand{\Id}{\textrm{Id}}

\begin{document}
\title{A Rigorous Computational Approach to Linear Response}
\author{Wael Bahsoun}
\address{Department of Mathematical Sciences, Loughborough University,
Loughborough, Leicestershire, LE11 3TU, UK}
\email{W.Bahsoun@lboro.ac.uk}
\author{Stefano Galatolo}
\address{Dipartimento di Matematica, Universita di Pisa, Via Buonarroti 1, Pisa - Italy}
\email{galatolo@dm.unipi.it}
\author{Isaia Nisoli}
\address{Instituto de Matematica - UFRJ Av. Athos da Silveira Ramos 149, Centro de Tecnologia - Bloco C Cidade Universitaria - Ilha do Fund\~ao. Caixa Postal 68530 21941-909 Rio de Janeiro - RJ - Brasil}
\email{nisoli@im.ufrj.br}
\author{Xiaolong Niu}
\address{Department of Mathematical Sciences, Loughborough University,
Loughborough, Leicestershire, LE11 3TU, UK}
\email{X.Niu@lboro.ac.uk}
\date{\today }

\pagestyle{myheadings} 
\markboth{Linear Response}{W.
Bahsoun, S. Galatolo, I. Nisoli, X. Niu}



\thanks{This work was mainly conducted during a visit of SG to Loughborough University. WB and SG would like to thank The Leverhulme Trust for supporting mutual research visits through the Network Grant IN-2014-021. SG thanks the Department of Mathematical Sciences at Loughborough University for hospitality. WB thanks Dipartimento di Matematica, Universita di Pisa. The research of SG and IN is partially supported by EU Marie-Curie IRSES ``Brazilian-European partnership in Dynamical Systems" (FP7-PEOPLE-2012-IRSES 318999 BREUDS). IN would like to thank the Department of Mathematics at Uppsala University and the support of the KAW grant 2013.0315}

\begin{abstract}
We present a general setting in which the formula describing the linear response of the physical measure of a perturbed system can be obtained. In this general setting we obtain an algorithm to rigorously compute the linear response. We apply our results to expanding circle maps. In particular, we present examples where we compute, up to a pre-specified error in the $L^{\infty}$-norm, the response of expanding circle maps under stochastic and deterministic perturbations. Moreover, we present an example where we compute, up to a pre-specified error in the $L^1$-norm, the response of the intermittent family at the boundary; i.e., when the unperturbed system is the doubling map. 
\end{abstract}

\maketitle

\section{Introduction}
A question of central interest from both theoretical and applied points of view in dynamical systems is the following: given a deterministic dynamical system that admits a Sinai-Ruelle-Bowen (SRB) measure, how does the SRB measure change if the original system gets perturbed, perhaps randomly? It is known that in certain situations the SRB measure changes smoothly and a formula of such a ``derivative" can be obtained \cite{Ba1, BaSma, D, GL, KP, R}. This is called the Linear Response formula. We refer to \cite{Ba2} for a recent survey about this area of research and to the most recent articles on linear response for intermittent maps \cite{BahSau,BT,K}. From a rigorous computational point of view there are no results in the literature that approximate the response of an SRB measure up to a pre-specified error in a suitable topology.

\bigskip

Our goal in this paper is to pioneer this direction of research and to provide tools to investigate the changes in the statistical properties of families of systems. Applications may range from the identification of tipping points in the statistical behavior of systems studied in applications, such as the ones considered in \cite{Lu}, to checking whether a family of systems has decreasing or increasing entropy, see for example the problems considered in \cite{CT} and their relation to number theory.

\bigskip

Our computational approach is based on finding a suitable finite rank approximation of the transfer operator associated with the original system. Such techniques have proved to be computationally robust and to be successful when approximating SRB measures of uniformly expanding systems \cite{BB, GN, Li1, M2}, (piecewise) uniformly hyperbolic systems \cite{Froy1, GN0}, and one-dimensional non-uniformly expanding maps \cite{BBD, GN, RM}. It has also proved to be a successful approach in approximating spectral data \cite{B, DJ, Froy2, Froy3, GNS, Li1} and limiting distributions of dynamical systems \cite{B1}. 

\bigskip

In this paper we show that suitable discretization schemes can be used to approximate linear response. The problem that we face in our rigorous approximation is two-fold. The first is functional analytic. In particular, we need to find suitable discretization schemes that preserve the regularity of the function space(s) where the transfer operator acts, and which can approximate the original transfer operator. The second is computational. In particular, the computational approach should be amenable to tracking all the round-off errors made by the computer.
 
\bigskip 

In Section \ref{sec1} we present a general setting in which the formula corresponding to the linear response can be obtained.
In this section we also show how the formula of such derivative can be rigorously computed using a computer.
In Section \ref{S3} we show how the algorithm can be implemented in the case of circle expanding maps.
In particular we find suitable discretization schemes and suitable Banach spaces achieving the goal for such maps.  
In Section \ref{noise} we apply our results to stochastic perturbations of expanding circle maps and we present an example where we compute, up to a pre-specified error  in the $C^0$ topology, the linear response of an expanding circle map under stochastic perturbations. In Section \ref{sec:deterministic}  we apply our results to a deterministic perturbation of an expanding circle map. In this example the exact response can be computed analytically. Thus, a comparison between the exact response and the computed one can be done. In Section \ref{appendix I} we present an example where we compute, up to a pre-specified error in the $L^1$-norm, the response of the intermittent family at the boundary; i.e., when the unperturbed system is the doubling map.  
Section \ref{appendix II} is an appendix that includes proofs and tools used in the computations in the examples of Sections \ref{noise} and \ref{sec:deterministic}.

\section{A general framework for the linear response \label{sec1}}  Let $X$ be a compact manifold with boundary. $X$ is the space on which we consider certain dynamics that we are going to slightly perturb and study the change of stationary measures after perturbation. The dynamics is described by the action of some positive transfer operator acting on some space of regular Borel measures on $X$. 
In this section we present a general setting in which the formula corresponding to the derivative of a fixed point\footnote{In applications to dynamical systems, such a fixed point corresponds to the density of an absolutely continuous invariant measure, or in general to a physical invariant measure.} of a family of such positive
operators $\mathcal{L}_{\epsilon }$ can be obtained\footnote{The differentiation is done with respect to the variable $\eps$ in a suitable norm. This will be clear in the statement of Proposition \ref{prop1} below.}.
We consider the action of the operators on three Banach spaces $ B_{ss} \subseteq B_s \subseteq B_w$ which are subsets of the space of finite  signed  Borel measures on $X$ and equipped with norms $||\cdot||_{w},||\cdot||_{s},||\cdot||_{ss}$ respectively, such that $||\cdot||_{w}\leq
||\cdot||_{s}\leq ||\cdot||_{ss}$. We suppose that $\mathcal{L}_{\epsilon },$ $\epsilon \geq 0,$ maps probability measures to probability measures and has a unique fixed point, which is a probability measure, $h_{\epsilon }\in B_{ss}.$ Let $\mathcal{L:=L}_{0}$ be the
unperturbed operator and $h\in B_{ss}$ be its fixed probability measure. Let $%
V_{s}^{0}=\{v\in B_{s},v(X)=0\},$ $V_{w}^{0}=\{v\in B_{w},v(X)=0\}.$ We assume that $V_s^0$ is closed in $B_s$.

\bigskip

The following proposition is essentially proved in \cite{Li2}. Since we adapted the assumptions to a general setting we include a proof.
\begin{proposition}
\label{prop1} Suppose that the following assumptions hold:
\begin{enumerate}
\item The norms $||\mathcal{L}^{k}||_{B_{w}\rightarrow B_{w}}$ and $||%
\mathcal{L}_{\epsilon }^{k}||_{B_{w}\rightarrow B_{w}}$ are uniformly bounded with respect to $k$ and $\epsilon >0$.
\item $\mathcal{L}_{\epsilon }$ is a perturbation of $\mathcal{L}$ in the
following sense
\begin{equation}
||\mathcal{L}_{\epsilon }-\mathcal{L}||_{B_{s}\rightarrow B_{w}}\leq
C\epsilon .  \label{approx}
\end{equation}
\item The operators $\mathcal{L}_{\epsilon }$, $\eps\ge0$, have uniform rate of contraction on $%
V_{s}^{0}$: there are $C_{1}>0$, $0<\rho <1$, such that 
\begin{equation}
||\mathcal{L}_{\epsilon }^{n}||_{V_{s}^{0}\rightarrow B_{s}}\leq C_{1}\rho
^{n}.  \label{unicontr}
\end{equation}
\item There is an operator $\mathcal{\hat{L}}:B_{ss}\to B_{s}$ such
that
\begin{equation}
\lim_{\epsilon \rightarrow 0}||\epsilon ^{-1}(\mathcal{L}_{\epsilon}-\mathcal{L})h-\mathcal{\hat{L}}h||_{s}=0 .  \label{Lhat} 
\end{equation}%
\end{enumerate}
Let 
\begin{equation*}
\hat{h}=(\Id-\mathcal{L})^{-1}\hat{\mathcal{L}}h.
\end{equation*}%
Then%
\begin{equation*}
\lim_{\epsilon \rightarrow 0}||\epsilon ^{-1}(h_\epsilon-h)-\hat{h}%
||_{w}=0;
\end{equation*}%
i.e. $\hat{h}$ represents the derivative of $h_{\epsilon }$ with respect to $\epsilon $.
\end{proposition}
\begin{proof}
Notice that since $V_s^0$ is closed in $B_s$ and our operators preserve probability measures, then $\hat{
\mathcal{L}}h\in V_{s}^{0}$. Recall that $h_{\epsilon }=\mathcal{L}
_{\epsilon }h_{\epsilon }$. We have
\begin{equation*}
(\Id-\mathcal{L}_{\epsilon })(h_{\epsilon }-h)=(\mathcal{L}_{\epsilon }-
\mathcal{L})h
\end{equation*}%
and since $\hat{h}=(\Id-\mathcal{L})^{-1}\mathcal{\hat{L}}h$, we obtain
\begin{equation} \label{proof1}
\begin{split}
\lim_{\epsilon \rightarrow 0}||\epsilon ^{-1}(h_{\epsilon }-h)-\hat{h}||_{w}
&=\lim_{\epsilon \rightarrow 0}||\epsilon ^{-1}(\Id-\mathcal{L}_{\epsilon
})^{-1}(\mathcal{L}_{\epsilon }-\mathcal{L})h-(\Id-\mathcal{L})^{-1}\mathcal{
\hat{L}}h||_{w} \\
&\leq \lim_{\epsilon \rightarrow 0}||(\Id-\mathcal{L}_{\epsilon
})^{-1}[\epsilon ^{-1}(\mathcal{L}_{\epsilon }-\mathcal{L})h-\mathcal{\hat{L}}h]||_{w} \\
&\hskip 1cm+\lim_{\eps\to 0}||(\Id-\mathcal{L}_{\epsilon })^{-1}\mathcal{\hat{L}}h-(\Id-\mathcal{L})^{-1}\mathcal{\hat{L}}h||_{w}\\
&:=(I)+(II).
\end{split}
\end{equation}
Notice that by assumption (3), $||(\Id-\mathcal{L}_{\epsilon })^{-1}||_{V_{s}^{0}\rightarrow B_{w}}$ are
uniformly bounded. Moreover, since $\lim_{\epsilon \rightarrow 0}||\epsilon ^{-1}(\mathcal{L}_{\epsilon }-
\mathcal{L})h-\mathcal{\hat{L}}h||_{s}=0,$ we obtain
\begin{equation*}
(I)=\lim_{\epsilon \rightarrow 0}||(\Id-\mathcal{L}_{\epsilon })^{-1}[\epsilon
^{-1}(\mathcal{L}_{\epsilon }-\mathcal{L})h-\mathcal{\hat{L}}h]||_{w}=0.
\end{equation*}
Now we consider $(II)$. By assumption (3), on the space $V_{s}^{0},\ (\Id-\mathcal{L}_{\epsilon })^{-1}=\sum_{0}^{\infty }%
\mathcal{L}_{\epsilon }^{k}$. Notice that by assumptions (2) and (3) we have:%
\begin{eqnarray*}
||\mathcal{L}^{k}-\mathcal{L}_{\epsilon }^{k}||_{V_{s}^{0}\rightarrow
V_{w}^{0}} &\leq &\sum_{j=0}^{k-1}||\mathcal{L}_{\epsilon }^{j}(\mathcal{L}%
_{\epsilon }-\mathcal{L})\mathcal{L}^{k-1-j}||_{V_{s}^{0}\rightarrow
V_{w}^{0}} \\
&\leq &\sup_{j}||\mathcal{L}_{\epsilon }^{j}||_{w}\sum_{j=0}^{k-1}||(%
\mathcal{L}_{\epsilon }-\mathcal{L})\mathcal{L}^{k-1-j}||_{V_{s}^{0}%
\rightarrow V_{w}^{0}} \\
&\leq &C\epsilon \sup_{j}||\mathcal{L}_{\epsilon }^{j}||_{w}\sum_{j=0}^{k-1}||%
\mathcal{L}^{k-1-j}||_{V_{s}^{0}} \\
&\leq &C\epsilon \sup_{j}||\mathcal{L}_{\epsilon }^{j}||_{w}C_{1}\frac{1-\rho
^{k}}{1-\rho }.
\end{eqnarray*}
Consequently,
\begin{equation*}
\begin{split}
&||(\Id-\mathcal{L}_{\epsilon })^{-1}\mathcal{\hat{L}}h-(\Id-\mathcal{L})^{-1}
\mathcal{\hat{L}}h||_{w}\\
 &\leq ||\mathcal{\hat{L}}h||_{s}[\sum_{k=0}^{l-1}||\mathcal{L}^{k}-\mathcal{L}_{\epsilon}^{k}||_{V_{s}^{0}\rightarrow V_{w}^{0}} +\sum_{l}^{\infty }(||\mathcal{L}^{k}||_{V_{s}^{0}\rightarrow V_{w}^{0}}+||\mathcal{L}_{\epsilon }^{k}||_{V_{s}^{0}\rightarrow V_{w}^{0}})] \\
&\hskip 1cm\leq ||\mathcal{\hat{L}}h||_{s}[C l \epsilon \sup ||\mathcal{L}_{\epsilon
}^{j}||_{w}C_{1}\frac{1-\rho ^{l}}{1-\rho }+2C_{1}\rho ^{l}\frac{1}{1-\rho }%
].
\end{split}
\end{equation*}
Choosing $l=\left\lceil |\log \epsilon |\right\rceil$ implies%
\begin{equation*}
(II)=\lim_{\epsilon \rightarrow 0}||(\Id-\mathcal{L}_{\epsilon })^{-1}\mathcal{%
\hat{L}}h-(\Id-\mathcal{L})^{-1}\mathcal{\hat{L}}h||_{w}=0.
\end{equation*}
Hence, $\lim_{\epsilon \rightarrow 0}||\epsilon ^{-1}(h_{\epsilon }-h)-\hat{h}%
||_{w}=0$.
\end{proof}
The function $h \to \hat{\mathcal{L}}h$ depends on the kind of perturbation we
consider (deterministic, stochastic, etc.). In the following, we suppose that $\hat{\mathcal{L}}h$ is computable with a small error in the $B_{w}$
norm. Then we show that this leads to the rigorous computation of $\hat{h}$ in the $
B_{w}$ norm. The computation will be performed by approximating $\mathcal{L}$ with a
finite rank operator $\mathcal{L}_{\eta }$ which can be implemented on a
computer. Let us consider a finite rank discretization%
\begin{equation*}
\Pi _{\eta }:B_{s}\rightarrow W_{\eta },
\end{equation*}%
where $W_{\eta }\subseteq B_{s} $ is a finite dimensional space of measures, such that for $
f\in B_{s}$, 
$$\lim_{\eta \rightarrow 0}||(\Pi _{\eta }-\Id)f||_{w}=0.$$ 
The operator $\mathcal{L}_{\eta }$ is defined as 
\begin{equation*}
\mathcal{L}_{\eta }=\Pi _{\eta }\mathcal{L}\Pi _{\eta }.
\end{equation*}%
Let us denote by $f_{\eta }\in V_{s}^{0}$ a family of approximations of $%
\hat{\mathcal{L}}h$ in the weak norm $||\cdot||_w$.

\begin{theorem} \label{main}Suppose that $\mathcal{L}$ satisfies the assumptions in Proposition \ref{prop1} and :
\begin{enumerate}
\item $||f_{\eta }||_{s}$ are uniformly bounded and $||f_{\eta }-\hat{%
\mathcal{L}}h||_{w}\underset{\eta \rightarrow 0}{\rightarrow }0.$

\item $\mathcal{L}_{\eta }$ is an approximation of $\mathcal{L}$ in the
following sense%
\begin{equation*}
||\mathcal{L}_{\eta }-\mathcal{L}||_{B_{s}\rightarrow B_{w}}\leq C\eta .
\end{equation*}

\item $\exists C>0$ such that for any $n\ge 1$ and $\eta>0$ we have $||\mathcal{L}^n_{\eta }||_s<C$.

 
\end{enumerate}
Then, for any $\tau >0$,  there are $\eta >0$ and $l^{\ast }\in \mathbb{N}$
such that 
\begin{equation*}
||(Id-\mathcal{L})^{-1}\hat{\mathcal{L}}h\ -\sum_{k=0}^{l^{\ast }-1}\mathcal{%
L}_{\eta }^{k}f_{\eta }||_{w}<\tau .
\end{equation*}
\end{theorem}
\begin{proof}
Notice that $(\Id-\mathcal{L})^{-1}\hat{\mathcal{L}}h$ is well defined
since $\hat{\mathcal{L}}h$ is of zero average. We have
\begin{align}
\sum_{k=0}^{l^{\ast }-1}||(\mathcal{L}^{k}-\mathcal{L}_{\eta }^{k})f_{\eta
}||_{w}& \leq \sum_{k=0}^{l^{\ast }-1}\sum_{j=0}^{k-1}||\mathcal{L}^{j}(%
\mathcal{L}-\mathcal{L}_{\eta })\mathcal{L}_{\eta }^{k-1-j}f_{\eta }||_{w} 
\notag \\
& \leq M\sum_{k=0}^{l^{\ast }-1}\sum_{j=0}^{k-1}||(\mathcal{L}-\mathcal{L}%
_{\eta })\mathcal{L}_{\eta }^{k-1-j}f_{\eta }||_{w}  \notag \\
& \leq M||(\mathcal{L}-\mathcal{L}_{\eta })||_{B_{s}\rightarrow
B_{w}}\cdot \sum_{k=0}^{l^{\ast }-1}\sum_{j=0}^{k-1}||\mathcal{L}_{\eta
}^{k-1-j}f_{\eta }||_{B_{s}},
\end{align}%
where $M=\sup_{k}||\mathcal{L}^{k}||_{B_{w}\rightarrow B_{w}}$. Consequently, we obtain
\begin{equation}
\begin{split}
& ||(Id-\mathcal{L})^{-1}\hat{\mathcal{L}}h-\sum_{k=0}^{l^{\ast }-1}\mathcal{%
L}_{\eta }^{k}f_{\eta }||_{w}=||\sum_{k=0}^{\infty }\mathcal{L}^{k}\hat{%
\mathcal{L}}h-\sum_{k=0}^{l^{\ast }-1}\mathcal{L}_{\eta }^{k}f_{\eta }||_{w}
\\
& \leq ||\sum_{k=l^{\ast }}^{\infty }\mathcal{L}^{k}\hat{\mathcal{L}}%
h||_{w}+||\sum_{k=0}^{l^{\ast }-1}\mathcal{L}^{k}\hat{\mathcal{L}}%
h-\sum_{k=0}^{l^{\ast }-1}\mathcal{L}_{\eta }^{k}f_{\eta }||_{w} \\
& \leq ||\sum_{k=l^{\ast }}^{\infty }\mathcal{L}^{k}\hat{\mathcal{L}}%
h||_{w}+\sum_{k=0}^{l^{\ast }-1}||(\mathcal{L}^{k}-\mathcal{L}_{\eta
}^{k})f_{\eta }||_{w}+\sum_{k=0}^{l^{\ast }-1}||\mathcal{L}^{k}(\hat{%
\mathcal{L}}h-f_{\eta })||_{w} \\
& \leq ||\sum_{k=l^{\ast }}^{\infty }\mathcal{L}^{k}\mathcal{\hat{L}}h||_{w}+M||(%
\mathcal{L}-\mathcal{L}_{\eta })||_{B_{s}\rightarrow B_{w}}\cdot
\sum_{k=0}^{l^{\ast }-1}\sum_{j=0}^{k-1}||\mathcal{L}_{\eta }^{k-1-j}f_{\eta
}||_{s} \\ 
&\hskip 7cm+M~l^{\ast }||\hat{\mathcal{L}}h-f_{\eta }||_{w}.
\end{split}
\label{poop}
\end{equation}
Now, choose $l^{\ast }$ big enough so that $||\sum_{k=l^{\ast }}^{\infty }%
\mathcal{L}^{k}\hat{\mathcal{L}}h||_{w}\leq \frac{\tau }{2}$. Since for each 
$\eta$ $||\mathcal{L}_{\eta }^{k-1-j}f_{\eta }||_{s}$ are
uniformly bounded, by assumptions (2) and (3) we can choose $\eta $ small enough such that%
\begin{equation}
M||(\mathcal{L}-\mathcal{L}_{\eta })||_{B_{s}\rightarrow B_{w}}\cdot
\sum_{k=0}^{l^{\ast }-1}\sum_{j=0}^{k-1}||\mathcal{L}_{\eta }^{k-1-j}f_{\eta
}||_{s}+M~l^{\ast }||\hat{\mathcal{L}}h-f_{\eta }||_{w}<\frac{\tau }{2}.
\label{eqlast}
\end{equation}%
\end{proof}
\begin{remark}\label{comput} For computational purposes it is important to have an algorithm
to find  suitable $l^{\ast }$ and  $\eta $. Let us comment on each summand in
Equation (\ref{poop}):
\begin{enumerate}
\item The first summand of (\ref{poop}), $||\sum_{k=l^{\ast }}^{\infty }%
\mathcal{L}^{k}\mathcal{\hat{L}}h||_{w}$ can be estimated by (\ref{unicontr}%
). However, it is enough to have an estimation on the weak norm. In Subsection \ref{sec:GNS} we will see how to find in systems satisfying a Lasota-Yorke
inequality, constants $C_{3},\rho _{3}$  such
that: $||\mathcal{L}^{k}\mathcal{\hat{L}}h||_{w}\leq C_{3}\rho _{3}^{k}||%
\mathcal{\hat{L}}h||_{s}$. Once the constants are found, we can bound $%
||\sum_{k=l^{\ast }}^{\infty }\mathcal{L}^{k}\mathcal{\hat{L}}h||_{w}\leq 
\frac{C_{2}\rho _{2}^{l^{\ast }}||\mathcal{\hat{L}}f||_{s}}{1-\rho _{2}}$
and find a suitable $l^{\ast }$ to make this summand as small as wanted.

\item For the second summand of (\ref{poop})%
\begin{equation*}
M||(\mathcal{L}-\mathcal{L_\eta})||_{B_{s}\rightarrow B_{w}}\cdot
\sum_{k=0}^{l^{\ast }-1}\sum_{j=0}^{k-1}||\mathcal{L}_{\eta }^{k-1-j}f_{\eta
}||_{s}
\end{equation*}
we need an estimate on $M$ which can be recovered by a Lasota-Yorke
inequality (see Proposition \ref{LYall} ). $||(\mathcal{L}-\mathcal{L}_{\eta })||_{B_{s}\rightarrow B_{w}}$
will be estimated by condition (2) of Theorem \ref{main}. The summands $||\mathcal{L}_{\eta }^{k-1-j}f_{\eta }||_{s}$ can
be approximated by the fact that $\mathcal{L}_{\eta }$ is of finite rank; i.e., by computing the matrix representing it. $||\mathcal{L}_{\eta }^{k-1-j}f_{\eta }||_{s}$ will be
estimated by the computer.

\item For $M~l^{\ast }||\mathcal{\hat{L}}h-f_{\eta }||_{w}$
\ of (\ref{poop}), we have to find a suitable approximation of $\mathcal{\hat{L}%
}h$ such that $||\mathcal{\hat{L}}h-f_{\eta }||_{w}$ is as small as wanted.
Note that this depends on the properties of $\hat{\mathcal{L}}$ and consequently on the kind of perturbations $\mathcal L_{\epsilon}$ that we consider. 
\end{enumerate}
\end{remark}

In the following we will discuss in details how the above results
can be applied to $C^3$ expanding maps of the circle. We also present
examples on how the algorithm outlined in Theorem \ref{main} and Remark
\ref{comput} can be implemented in this setting. The concrete
implementation of the above ideas to expanding maps of the circle
involves spaces of measures having a smooth density. The general framework extends to more general classes of hyperbolic systems,
provided a suitable functional analytic framework is considered (see
\cite{bani} for a recent survey on suitable spaces to be considered for
hyperbolic systems). The implementation to such systems is out of the
scope of the current paper.

\section{Circle  expanding maps and smooth discretizations}\label{S3}

Theorem \ref{main} and Remark \ref{comput}  outline an algorithm for the computation of the linear response.

\begin{enumerate} 

\item First we have to find suitable $f_{\eta}$ approximating $\hat{\mathcal{L}}h$ in the weak norm.

\item Then we can use a suitable discretization $\mathcal{L}_{\eta}$ of the transfer operator (well approximating it as an operator from $\mathcal{B}_s$ to $\mathcal{B}_w$) for the computation of  $l^*$ as in Item (1) of Remark \ref{comput}. 

\item Once found the suitable $l^*$  we compute the result of our approximation procedure as $$\hat{h}_{appr}=\sum_{i=0}^{l^*-1} \mathcal{L}_{\eta}^i f_{\eta}. $$
\end{enumerate} 

If $f_{\eta}$, $\mathcal{L}_{\eta}$,  $l^*$ are well chosen, Theorem \ref{main} ensures that $\hat{h}_{appr}$ is a good approximation of $\hat{h}$
in the weak norm and by Remark \ref{comput} we can explicitly bound the approximation error.
In this section we describe suitable functional spaces and a good approximation $\mathcal{L}_{\eta}$   for the transfer operator 
for expanding circle maps. This gives us an approximation of the linear response in $L^{\infty}$.
The approximation $f_{\eta}$ of $\hat{\mathcal{L}}h$ depends on the kind of perturbation considered.
In the following sections we will discuss two specific kinds of perturbations: deterministic and stochastic ones.

%

\bigskip

Let us consider the space $(\mathbb{T},\mathfrak{B},m)$ where $\mathbb{T}$ is
the unit circle, $\mathfrak{B}$ is Borel $\sigma $-algebra and $m$ is 
Lebesgue measure on $\mathbb{T}$. 
Let $T:\mathbb{T}\rightarrow \mathbb{T}$ be a $C^{3}$ uniformly expanding
circle map; i.e. $\inf_{x\in \mathbb{T}}|D_{x}T|>1$. Let $$\lambda =1/\inf_{x\in \mathbb{T}}|D_{x}T|.$$ 
Without loss of generality we assume that $T$ is orientation preserving. 
The circle map $T$ it is naturally associated to an 
expanding map $[0,1]\to [0,1]$ which will be still denoted by $T$.
We will consider different perturbations of  the transfer operator 
associated to this kind of maps and apply Theorem \ref{main} to compute the linear response $\hat{h}$.  
We will consider the action of the transfer operator on the function spaces 
$B_{w}=C^0([0,1])$, $B_{s}=C^1([0,1])$ and $B_{ss}=C^2([0,1])$. 
We equip the $C^k$ spaces, $k=0,1,2,$ with the usual norms $||f||_{C^k}=\sum_{i\leq k}||f^{(i)}||_{\infty} $.

It is known that such an expanding map  has a $C^2$ invariant density and there 
is an explicit formula for the action of the transfer operator 
associated with $T$ on probability densities (also called  Perron-Frobenius operator, 
see \cite{Ba}) $\mathcal{L}:C^{0}([0,1])\rightarrow C^{0}([0,1])$
\begin{equation}\label{PF}
\mathcal{L}f(x)=\sum_{y\in T^{-1}x}\frac{f(y)}{T^{\prime }(y)}.
\end{equation}

The reason behind working on the closed interval rather than 
the unit circle is that there are some advantages in the 
computer implementation of the discretization 
(the implementation on $[0,1]$ is easier and cleaner than the implementation on the circle). 
We can also consider our function spaces as spaces of smooth functions on the circle 
allowing discontinuities at $0$. In the following when there is no ambiguity we will denote these spaces by $C^0, C^1, C^2$.

\subsection{Basic properties of the transfer operator}\label{subsec:modfuncspace} 
We will  need some additional information on the action of the transfer operator on the space $C^2([0,1])$
to understand better the properties of its invariant density. Let 
\begin{equation*}
J(f)=|f(1)-f(0)|.
\end{equation*}%
Notice that $C^0(\mathbb T)$ is the set of functions $f$ in $C^0([0,1])$ such that
$J(f)=0$. In the next lemma we are going to prove that, indeed, the fixed point of $\mathcal{L}$
is in $C^0(\mathbb{T})$.
\begin{lemma}
\label{Jcontract} We have
\begin{enumerate}
\item $\mathcal{L}$ preserves $C^{k}([0,1])$, $k\in \{1,2\};$
\item $J(\mathcal{L}f)\leq \lambda J(f),$
\item if $h$ is a fixed point of $\mathcal{L}$ in $C^1([0,1])$, then $J(h)=0$.
\end{enumerate}
\end{lemma}
\begin{proof}
We suppose $T(0)=0$, by the regularity of $T$ and the form of the operator
follows that $\PF$ preserves $C^k([0,1])$.

To prove the second statement, let us
denote by $d_{i}$ the preimages of $0$ that are contained inside the
interval $(0,1)$. By continuity of $T$ on $(0,1)$ we have 
\begin{equation*}
\mathcal{L}f(0)-\mathcal{L}f(1)=\frac{1}{T^{\prime }(0)}(f(0)-f(1))+\sum_{i}\frac{f(d_{i})}{T^{\prime
}(d_{i})}-\sum_{i}\frac{f(d_{i})}{T^{\prime }(d_{i})}.
\end{equation*}
If $\mathcal{L}h=h$ then $J(h)\leq \lambda J(h)$. Since $\lambda<1$ this implies item $3$.
\end{proof}
Before introducing our discretization scheme, we state Lasota-Yorke inequalities for $\mathcal{L}$ when acting on ${C}^{1}([0,1])$, ${C}^{2}([0,1])$. 
Since these inequalities will be used in the computer implementation, we also give estimates for the constants involved. For the proof of Proposition \ref{LYall}, see Section \ref{sec:LY} in the Appendix \ref{appendix II}. 
\begin{proposition}\label{LYall} 
\end{proposition}
\begin{enumerate}
\item Let $M:=\sup_{n}||\mathcal{L}^{n}||_{L^{\infty }\rightarrow L^{\infty
}}$. Then 
$$M\leq 1+\frac{B}{1-\lambda },$$
where $\lambda :=(\inf_{x\in \mathbb{T}}|D_{x}T|)^{-1}<1$ and $B=||T^{\prime
\prime }/(T^{\prime 2})||_{\infty }.$\\

\item For $f\in {C}^{1}([0,1])$ we have
\begin{equation*}
||\mathcal{L}f||_{{C}^{1}}\leq \lambda M||f||_{{C}%
^{1}}+M^2||f||_{\infty }.
\end{equation*}%

\item For $f\in {C}^{2}([0,1])$ we have
\begin{equation*}
||\mathcal{L}f||_{{C}^{2}}\leq \lambda ^{2}M||f||_{{C}%
^{2}}+D||f||_{{C}^{1}}.
\end{equation*}%
where%
\begin{equation*}
D=\lambda M+M^2+3\max \{1,\bigg|\bigg|\frac{T^{\prime \prime }}{(T^{\prime
})^{2}}\bigg|\bigg|_{\infty }^{2}\}M+M\bigg|\bigg|\frac{T^{\prime \prime
\prime }}{(T^{\prime })^{3}}\bigg|\bigg|_{\infty }.
\end{equation*}
\end{enumerate}

The above inequalities, along with the properties of the system, imply that $
\mathcal{L}$ has a spectral gap on ${C}^{1}([0,1])$ and on $%
{C}^{2}([0,1]).$ Moreover, $1$ is a simple dominant eigenvalue.
In particular, this implies that $T$ admits a unique invariant density $h$ in $
{C}^{2}([0,1])$ and the system $(T,\mathbb{T},\mu )$, where $\mu
:=h\cdot m$, is mixing (see \cite{notes} for an elementary proof of this).

\subsection{A finite rank approximation of $\mathcal L$  as an operator from $C^1\to C^0$.}\label{newproj}
To compute the rate of convergence to equilibrium and the linear response we  introduce a finite rank
approximations of $\mathcal L$ which will be called $\mathcal{L}_{\eta }$.

\bigskip

We start by defining a suitable partition of unity. 
Let us consider the partition of unity $\{\phi _{i}\}_{i=0}^{m}$  
defined in the following way: for $i=0,\dots,m$, let $a_{i}=i/m$. For $i=0,\ldots ,m$ set
\begin{equation}\label{unity}
\phi_i(x)=\phi(m\cdot x-i),
\end{equation}
where 
\begin{equation}
\phi(x)=\left\{ 
\begin{array}{cc}
1-3x^2-2x^3 & x\in \lbrack -1,0] \\ 
1-3x^2+2x^3 & x\in \lbrack 0,1]\\
0&\text{otherwise}
\end{array}%
\right..
\end{equation}%
Note that for $i=0$ and $i=m$, the bump function is restricted to half of its support. Also note that $\phi_i(a_j)=\delta_{ij}$ (where $\delta_{ij}=1$ if $i=j$, $0$ in all the other cases) 
and that $||\phi_i||_{\infty}=1$, $||\phi_i'(x)||_{\infty}=3m/2$, $\sum_i \phi_i(x)=1$.
\begin{remark}
There are reasons why this choice of $\phi$ is sensible for our line of work: computing the value of
a cubic polynomial is fast by using Horner's scheme, and rigorous bounds are implemented via interval arithmetics \cite{Tucker}. The same is true for the derivatives of
the $\phi_i$ defined above\footnote{An alternative approach would be to choose a smooth bump function
\begin{equation}
\phi(x)=\left\{ 
\begin{array}{cc}
e^{-\frac{1}{1-x^2}+1} & |x|<1  \\ 
0 & |x|>1%
\end{array}%
\right..
\end{equation}%
and build a partition of unity by rescaling and translating this function, but the implementation
of this approach is more delicate since the derivative of $\phi$ cannot be implemented in a naive way and the sum 
of two functions 
\[\phi(m\cdot x-i)+\phi(m\cdot x-(i+1))\neq 1\] 
for $x\in [a_i,a_{i+1}]$.}.
\end{remark}

To ensure that our discretization preserves integrals, we use an auxiliary function $\kappa(x)$:
\[
\kappa(x)=6x(1-x),
\]
by direct computation $||\kappa(x)||_{\infty}=3/2$, $||\kappa'(x)||_{\infty}=6$. Moreover, $\int_0^1 \kappa(x)dx=1$.

Set $\eta:=1/m$ and define
\begin{equation*}
\Pi _{\eta }(f)(x):=\sum_{i}f(a_{i})\cdot \phi _{i}(x)+\bigg(\int_0^1 f dm -\sum_{i=0}^m f(a_i)\int_0^1 \phi_i dm \bigg)\kappa(x).
\end{equation*}
We set
\begin{equation}\label{myop}
\mathcal{L}_{\eta }:=\Pi _{\eta }\mathcal{L}\Pi _{\eta }.
\end{equation}
We now prove properties of $\Pi_{\eta}$ that will be used to verify the assumptions of Theorem \ref{main}.
\begin{lemma}
\label{it:C2contr} For $f\in {C}^{1}([0,1])$, we have 
\begin{enumerate} 
\item $||\Pi _{\eta }f||_{\infty}\le 4||f||_{\infty}$;
\item $||\Pi _{\eta }f||_{\infty}\le ||f||_{\infty}+\frac{3}{2}||f'||_{\infty}\eta$;
\item $||(\Pi _{\eta }f)'||_{\infty}  \leq (\frac{3}{2}+6\eta )||f'||_{\infty}$;
\item $||\Pi _{\eta }f-f||_{\infty }\leq \frac{5}{2}||f^{\prime}||_{\infty }\eta$.
\end{enumerate}
\end{lemma}
\begin{proof}
The following approximation inequality holds:
\begin{equation}\label{e40}
\begin{split}
|f(x)-\sum f(a_{i})\phi _{i}(x)|& =|\sum_{i}(f(x)-f(a_{i}))\phi _{i}(x)| \\
& =|\sum_{i}f^{\prime }(\xi_i )(x-a_{i})\phi _{i}(x)|\leq 
||f^{\prime}||_{\infty }\eta.
\end{split}
\end{equation}
This implies that
\begin{equation}\label{wb:eq1}
\bigg|\int_0^1 f dm -\sum_{i=0}^m f(a_i)\int_0^1 \phi_i dm\bigg|\leq 
||f^{\prime}||_{\infty }\eta. 
\end{equation}
By \eqref{wb:eq1}, we have
\begin{align}\label{eq:infbound}
||\Pi _{\eta }f||_{\infty }&\leq \max_{x\in[0,1]}|\sum_{i=0}^n f(a_i)\phi_i(x)|+|\int f(x)-\sum_{i=0}^n f(a_i)\phi_i(x) dx|\cdot ||\kappa(x)||_{\infty}\\ \nonumber
&\leq ||f||_{\infty}|\sum_i \phi_i(x)|+\frac{3}{2}|\int f(x)-\sum_{i=0}^n f(a_i)\phi_i(x) dx|, 
\end{align}
which implies (1) and (2) of the lemma. We now prove (3). First, since the $\{\phi_i\}_{i=0}^m$ is a partition of unity, we have
\[
\sum_{i=0}^n \phi_i'(x)=0.
\]
Therefore,
\begin{align}\label{eq:derbound}
||(\Pi_{\eta}f)'||_{\infty}&\leq|\sum_{j=0}^{m/2} f(a_{2j})-f(a_{2j+1})\phi_i'(x)|+|\int f(x)-\sum_{i=0}^n f(a_i)\phi_i(x) dx|\cdot ||\kappa'(x)||_{\infty}\\ \nonumber
&\leq \frac{3}{2}\max_{j}|\frac{f(a_{2j})-f(a_{2j+1})}{\eta}|+6||f'||_{\infty}\eta\leq \bigg(\frac{3}{2}+6\eta\bigg)||f'||_{\infty}. 
\end{align}
Thus, (3) follows from \eqref{eq:infbound}, \eqref{eq:derbound}.  Also note that (4) of the lemma follows from \eqref{e40} and \eqref{wb:eq1}.
\end{proof}

\begin{remark}\label{assumptmain} By Item (4) of Lemma \ref{it:C2contr}, assumption (2) of Theorem \ref{main} is satisfied. By the Lasota-Yorke inequalities\footnote{See also the Appendix \ref{appendix II} for a proof of a uniform Lasota-Yorke inequality of $\mathcal L$ and $\mathcal L_{\eta}$ on $C^1$.}
given in Proposition \ref{LYall},  
$\mathcal{L}$ and $\mathcal{L}_{\eta }$ satisfy assumption (3) of Theorem \ref{main}. 
\end{remark}
\section{Response for a stochastic perturbation}\label{noise}
In this section we consider  stochastic perturbations of the expanding maps described in the previous section. 
At each step we add a small random perturbation distributed with a certain probability density $j$. 
We describe the analytic estimates which are necessary to apply our algorithm in that case and show the 
result of an actual implementation, where we compute the response for the stochastic perturbation of an expanding map. 
In particular we show the existence and the structure of the operator $\hat{\mathcal L}$ for  this case. We use $\text{Var}(\cdot)$ to denote the one dimensional variation of a function. 
\subsection{A stochastic perturbation}\label{Stoch}
For $f\in C^0([0,1])$ let $%
K_{\varepsilon }$ denote the operator defined as: 
\begin{equation*}
K_{\varepsilon }f(x)=\int_{\mathbb{T}}\varepsilon ^{-1}j(\varepsilon
^{-1}(x-y))f(y)dy,
\end{equation*}%
where $j\in C^{\infty }(\mathbb{R},\mathbb{R}^{+})$, $\text{supp}(j)\subset
\lbrack -1/2,1/2]$ and $\int_{\mathbb{R}}j(y)dy=1$.
\begin{lemma}
[Properties of $K_{\protect\epsilon }$]\label{lemma23}
\text{ }\\
\begin{enumerate}
\item For $f\in {C}^{k}$, $k\in \{0,1,2\}$%
\begin{equation*}
\Var(K_{\epsilon }f)\leq \Var(f) ,\quad ||K_{\epsilon }f||_{{C}^{k}}\leq ||f||_{{C}^{k}};
\end{equation*}

\item for $f\in {C}^{1}$%
\begin{equation*}
||K_{\epsilon }f-f||_{\infty }\leq \epsilon ||f||_{{C}^{1}};
\end{equation*}

\item for $f\in {C}^{2}$ 
\begin{equation*}
||K_{\epsilon }f-f||_{{C}^1}\leq \epsilon ||f||_{{C}^{2}}.
\end{equation*}
\end{enumerate}
\end{lemma}

\begin{proof}
The first assertion is a standard property of convolution.

For (2), we have
\begin{equation*}
|K_{\epsilon }f(x)-f(x)|=|\int_{\mathbb{T}}\epsilon ^{-1}j(\epsilon
^{-1}(x-y))(f(y)-f(x))dy|\leq \epsilon ||f^{\prime }||_{\infty }.
\end{equation*}%
Since the support of $j$ is contained in $[-1 /2,1 /2]$. To prove (3), observe that 
\begin{equation*}
\frac{\partial }{\partial x}j(\epsilon ^{-1}(x-y))=-\frac{\partial }{%
\partial y}j(\epsilon ^{-1}(x-y)).
\end{equation*}%
Therefore, 
\begin{equation*}
|(K_{\epsilon }f(x))^{\prime }-f^{\prime }(x)|=|\int_{\mathbb{T}}\epsilon
^{-1}j(\epsilon ^{-1}(x-y))\frac{\partial }{\partial y}f(y)dy-f^{\prime
}(x)|.
\end{equation*}%
Using integration by parts and the compactness of the support of $j$, we
obtain: 
\begin{equation*}
|(K_{\epsilon }f(x))^{\prime }-f^{\prime }(x)|\leq \epsilon ||f^{\prime
\prime }||_{\infty }.
\end{equation*}
\end{proof}
We now define the (average) transfer operator relative to the stochastic system with size $\epsilon$ noise  by setting
\begin{equation}\label{Kerop}
\mathcal{L}_{\varepsilon }:=K_{\varepsilon }\mathcal{L},
\end{equation}
where  $\mathcal{L}$ is the transfer operator for the map without noise defined in \eqref{PF}. 
Below we find the response of stochastic perturbation \eqref{Kerop} when $\epsilon $ is small. 
We start by finding the formula of the corresponding operator $\mathcal{\hat{L}}$.
\begin{proposition}
Let $\mathcal{L}_{\varepsilon }$ be as in \eqref{Kerop}. Set $ \gamma :=\int j(\xi )\xi d\xi$. For any $f\in C^{2}$
\begin{equation}
\lim_{\epsilon \to 0}||\frac{1}{\epsilon }(\mathcal{L}_{\epsilon }(f)-\mathcal{L(}f))-\gamma (\mathcal Lf)^{\prime}||_{C^1}=0.
\label{eq:gamma}
\end{equation} 
In particular, 
\[
\mathcal{\hat{L}}f=\gamma (\mathcal{L}f)^{\prime }.
\]
\end{proposition}
\begin{proof}
Notice that $\frac{1%
}{\epsilon }(\mathcal{L}_{\epsilon }(f)-\mathcal{L(}f))=\frac{1}{\epsilon }%
\left( K_{\epsilon }-Id)(\mathcal{L}f\right) $ and recall that for any $f \in C^{2}$ we have $%
\mathcal{L}f \in C^{2}$. Therefore, it is sufficient to prove that, for every $f\in C^{2}$, $\frac{1}{%
\epsilon }(K_{\epsilon }-Id)f$ converges to $\gamma f^{\prime }$ in the $%
C^{1}$ norm as $\epsilon \rightarrow 0$ . We recall
that the noise kernel is given by rescaling a fixed kernel $j$, that is $%
j_{\epsilon }(x)=\frac{1}{\epsilon }j(\frac{1}{\epsilon }x)$. Thus, the
support of $j_{\epsilon }$ is contained in the interval $[-\epsilon ,\epsilon
]$ and  $\int xj_{\epsilon }(x)dx=\epsilon \gamma $. To prove the convergence in $C^{1}$ of the limit, we have to show that%
\[
\lim_{\epsilon \rightarrow 0}||\frac{1}{\epsilon }\left(\int
j_{\epsilon }(t-x)f(t)~dt-f(x)\right) -\gamma f^{\prime }(x)||_{\infty}=0
\] and
\[
\lim_{\epsilon \rightarrow 0}||\frac{1}{\epsilon }\left(\int
j_{\epsilon }(t-x)f'(t)~dt-f'(x)\right) -\gamma f^{\prime\prime}(x)||_{\infty}=0.
\] 
The first limit can be treated as follows: 
\begin{gather*}
\begin{split}
&\left|\frac{1}{\epsilon }\left(\int j_{\epsilon }(t-x)f(t)~dt-f(x)\right) -\gamma f^{\prime }(x)\right|\\
&=\left|\frac{1}{\epsilon}\left( \int j_{\epsilon}(t-x)(t-x)\frac{f(t)-f(x)}{(t-x)}~dt\right) -\gamma f^{\prime }(x)\right| \\
&=\left|\frac{1}{\epsilon }\int j_{\epsilon}(t-x)(t-x)\left(\frac{f(t)-f(x)}{(t-x)}-f^{\prime }(x)\right)~dt\right| \\
&\leq \gamma\cdot\sup_{t\in [x-\epsilon ,x+\epsilon ]}\left|\frac{f(t)-f(x)}{(t-x)}-f^{\prime }(x)\right|.
\end{split}
\end{gather*}
Since $f \in C^{1}[0,1]$, $\sup_{t\in \lbrack x-\epsilon ,x+\epsilon ]}|(%
\frac{f(t)-f(x)}{(t-x)}-f^{\prime }(x))|\rightarrow 0$ uniformly. The  limit
\[
\lim_{\epsilon \rightarrow 0}||\frac{1}{\epsilon }\left(\int
j_{\epsilon }(t-x)f'(t)~dt-f'(x)\right) -\gamma f^{\prime\prime}(x)||_{\infty}
\] 
can be treated in the same way using the fact that $f\in C^{2}$.
\end{proof}
In order to apply Proposition \ref{prop1} and Theorem \ref{main} to the stochastic perturbation we note that the previous proposition ensures that assumption (4) of Proposition \ref{prop1} is satisfied. For the other
assumptions we refer to the following remark.
\begin{remark}
Let $h$ be the fixed probability density in $C^2([0,1])$ of $\mathcal{L}$. By Lemma \ref{Jcontract}
we have $J(h)=0$. Therefore,
\[
\int_0^1 h' dx=0.
\]
\end{remark}
\begin{remark}\label{RKL}
\label{remarkonstab} \ By Item (1) of Lemma \ref{lemma23} it follows that 
\[
\Var(K_{\epsilon}\mathcal{L}f)\leq \Var(\mathcal{L}f)\quad\text{and}\quad ||K_{\epsilon}\mathcal{L}f||_{C^k}\leq ||\mathcal{L}f||_{C^k}
\]
for $k=1,2$. Therefore, $\mathcal{L}_{\epsilon }$ and $\mathcal{L}$, satisfy uniform Lasota-Yorke
inequalities on the spaces $BV$, $C^1([0,1])$ and $C^2([0,1])$. 
This implies that assumption (1) of Proposition \ref{prop1} holds. \footnote{Let $f\in C^0$ be a density. Suppose $||f||_\infty \leq 1$ and consider $||\mathcal{L}_\epsilon^n f||_\infty$. First remark that since $\mathcal{L_\epsilon}$ is positive,  $f\geq 0$ implies $||\mathcal{L_\epsilon}^n f||_\infty \leq ||\mathcal{L_\epsilon}^n 1||_\infty $ ($1$ is the function having constant value $1$). Furthermore $||\mathcal{L_\epsilon}^n 1||_\infty$  is uniformly bounded by the Lasota Yorke inequalities. 
In the general case we can decompose $f$ into positive and negative parts and apply the same construction to each component. }
Moreover, by the stability result of \cite{KL} (see also \cite{notes} 
for an elementary proof of a similar result) for sufficiently small $\epsilon >0$,
assumption (3) of Proposition \ref{prop1} holds. Finally,  by Item (2) of Lemma \ref{lemma23} we obtain the approximation assumption (Item
(2)) of Proposition \ref{prop1}.
\end{remark}
Thus, by Proposition \ref{prop1}, the linear response holds:
\begin{equation}
\lim_{\epsilon \rightarrow 0}||\frac{h_{\epsilon }-h}{\epsilon }-%
\hat{h}||_{\infty }=0,  \label{response}
\end{equation}
where $\hat{h}:=(\Id-\mathcal{L})^{-1}\mathcal{\hat{L}}h.$
\subsection{The rigorous computation of the response $\hat{h}$}\label{3.5}
Now we show how to compute the linear response under stochastic perturbations for the class of systems described in this section. 

\bigskip

Assume we are given a family of $C^2$ functions\footnote{Note that $h_\eta$ is not the fixed point of the discretization $\mathcal L_\eta$ defined earlier in the section. In the example of this section we obtain the sequence $h_\eta$ through the discretization defined in  subsection \ref{subsec:approxLhatstochastic} in the Appendix \ref{appendix II}.} $h_\eta$ such that $||h_\eta-h||_{C^1}\to 0$ as $\eta\to 0$. In particular, this would imply
\[
||h'_{\eta}-h'||_{\infty}\to 0\quad\text{as}\quad\eta\to 0.
\]
We can then apply Theorem \ref{main} with $f_\eta:=\gamma h'_\eta$ to obtain:
\begin{corollary}
Assume $h_\eta$ is a $C^2$ family of functions such that $||h_\eta-h||_{C^1}\to 0$ as $\eta\to 0$. Then, given $\tau >0$, $\exists ~l^{\ast }\in \mathbb{N}$ and $\eta >0$ such that 
\begin{equation*}
||\hat{h}-\gamma \sum_{k=0}^{l^{\ast }-1}\mathcal{L}_{\eta }^{k}{h}_{\eta
}^{\prime }||_{\infty }<\tau ,
\end{equation*}%
where $h'_{\eta }$ is the derivative of $h_\eta$, $\gamma $ is as in equation \eqref{eq:gamma} and $\mathcal{L}_{\eta }$ is the operator defined in \eqref{myop}.
\end{corollary}
\begin{proof} Since $\mathcal{\hat{L}}f:=\gamma(\mathcal L f)'$ (see equation (\ref{eq:gamma})),  the proof is a direct application of Theorem \ref{main} with $f_{\eta }=\gamma{h}_{\eta }^{\prime }$ .  Assumption (1) of Theorem \ref{main} follows by $||h_\eta-h||_{C^1}\to 0$ as $\eta\to 0$. 
Recall that the remaining assumptions (2), (3) of Theorem \ref{main} are established in Remark \ref{assumptmain}.
\end{proof}
To estimate the rigorous error we have to find suitable $l^*$ and $\eta$. We follow Remark \ref{comput}.
If we denote by $\hat{h}_{appr}$ the approximation of the linear response we have that:
\begin{align*}
&||\hat{h}-\hat{h}_{appr}||_{\infty}\leq \\
&||\sum_{i=l^*}^{+\infty}\mathcal{L}^i\hat{\mathcal{L}}h||+||\sum_{i=0}^{l^*-1}\mathcal{L}^i f_{\eta}-\sum_{i=0}^{l^*-1}\mathcal{L}^i_{\eta}f_{\eta}||_{\infty}+||\sum_{i=0}^{l^*-1}\mathcal{L}^i (f_{\eta}-\hat{\mathcal{L}}h)||_{\infty}.
\end{align*}
To estimate $||\sum_{i=l^*}^{+\infty}\mathcal{L}^i\hat{\mathcal{L}}h||_{\infty}$ we use the uniform contraction,
whose coefficients can be estimated using the method in Subsection \ref{sec:GNS}; we can find $C_1$, $k$ and $\rho$
such that
\[
||\mathcal{L}^k h'||_{\infty}\leq C_1\rho||h'||_{C^1}\leq C_1\rho||h||_{C^2};
\]
therefore, if $l=k\cdot h$
\[
||\sum_{i=l}^{+\infty}\mathcal{L}^i h'||_{\infty}\leq C_1\cdot M\cdot k\frac{\rho^h}{1-\rho}||h||_{C^2},
\]
and $||h||_{C^2}$ can be bounded from the coefficients of the $C^2$ Lasota-Yorke inequality in Proposition \ref{LYall};
this permits us to find $l^*$. The second summand may be estimated by
\[
||\sum_{i=0}^{l^*-1}\mathcal{L}^i f_{\eta}-\sum_{i=0}^{l^*-1}\mathcal{L}^i_{\eta}f_{\eta}||_{\infty}
\leq M ||\mathcal{L}-\mathcal{L}_{\eta}||_{C^1\to C^0}\sum_{k=0}^{l^*-1}\sum_{j=0}^{k-1}||\mathcal{L}_{\eta}^{k-1-j}f_{\eta}||_{C^1}, 
\]
where we numerically estimate $||\mathcal{L}_{\eta}^{i}f_{\eta}||_{C^1}$. The third summand is estimated by:
\[
||\sum_{i=0}^{l^*-1}\mathcal{L}^i (f_{\eta}-\hat{\mathcal{L}}h)||_{\infty}\leq M\cdot \gamma\cdot  l^* ||h'-h_{\eta}'||_{\infty}.
\]
\subsection{An example of linear response under a stochastic perturbation}\label{sec4}
In this example we study a circle expanding map $T$
and the behavior of the density when, at each iteration step,
we add a noise, as explained in Subsection  \ref{Stoch}.
We consider $T:S^1\to S^1$:
\[
T(x)=8x+0.0025(\sin(16\pi x)+\frac{1}{4}\sin(32\pi x))\textrm{ mod $1$};
\]
the operator $\mathcal{L}$, associated to $T$, satisfies the following inequalities:
\begin{align*}
&\textrm{Var}(\mathcal{L} f)\leq 0.127\cdot \textrm{Var}(f)+0.2||f||_1;\\
&||\mathcal{L}^n||_{\infty}\leq M=1.2,\forall n\geq  0 ;\\
&||\mathcal{L}^k f||_{C^1}\leq 1.2\cdot 0.127^k ||f||_{C^1}+1.44||f||_{\infty} ;\\
&||\mathcal{L}^k f||_{C^2}\leq 1.2\cdot 0.017^k ||f||_{C^2}+3.8||f||_{C^1}.
\end{align*}
Let $h$ be the fixed point of $\mathcal{L}$ in ${C}^2$. Following Subsection \ref{Stoch} we have that
$\hat{\mathcal{L}}:{C}^2\to {C}^1$ is given by
\[
\hat{\mathcal{L}}h=\gamma h'
\]
and, by Proposition \ref{prop1}, the linear response is given by
\[
\hat{h}=(\Id-\mathcal{L}_0)^{-1}\hat{\mathcal{L}}h=\sum_{i=0}^{+\infty}\mathcal{L}^i \hat{\mathcal{L}}h.
\]
To compute the linear response we need to compute an approximation $f_{\eta}$ to $\gamma h'$; to do so we use the discretization 
in Subsection \ref{subsec:approxLhatstochastic}. Let us choose $l^*=18$ and $\eta=524288$ both for the approximation of the density in
$C^1$ and the computation of the linear response; we approximate the linear response by
\[
\hat{h}_{appr}=\sum_{i=0}^{17}\mathcal{L}_{\eta}f_{\eta}.
\]

\begin{figure}
  \begin{subfigure}[b]{0.45\textwidth}
    \includegraphics[width=55mm,height=50mm]{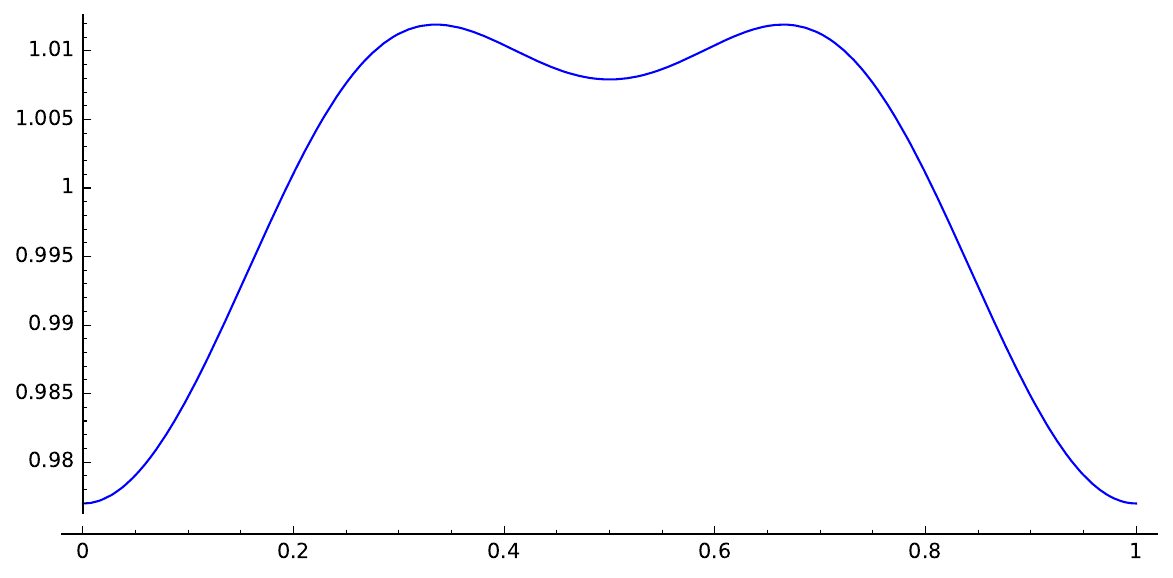}
   \caption{A plot of of the invariant density}
  \end{subfigure}
  \begin{subfigure}[b]{0.45\textwidth}
   \includegraphics[width=55mm,height=50mm]{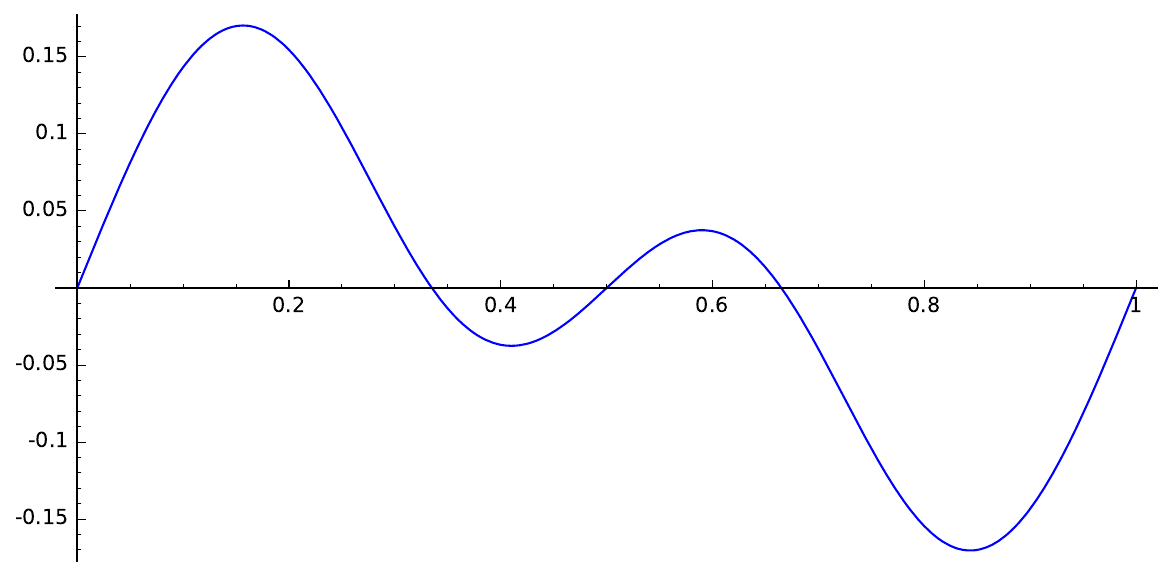} 
   \caption{The computed linear response}
  \end{subfigure}  
  \caption{The computed approximations of the response in the stochastic case, and a plot of an approximation of the invariant density for the non-perturbed map.}
  \label{fig:linear_response_stochastic}
\end{figure}

In figure \ref{fig:linear_response_stochastic} a plot of the approximations of the invariant density and of the linear response of the map under stochastic perturbation are presented. We are going to estimate the error using the algorithm developed in the present paper (refer to Theorem \ref{main} and the subsequent discussion):
\begin{align*}
&||\hat{h}-\hat{h}_{appr}||_{\infty}\leq \\
&||\sum_{i=l^*}^{+\infty}\mathcal{L}^i\hat{\mathcal{L}}h||_{\infty}+||\sum_{i=0}^{l^*-1}\mathcal{L}^i f_{\eta}-\sum_{i=0}^{l^*-1}\mathcal{L}^i_{\eta}f_{\eta}||_{\infty}+||\sum_{i=0}^{l^*-1}\mathcal{L}^i (f_{\eta}-\hat{\mathcal{L}}h)||_{\infty}.
\end{align*}
In the following subsections we are going to estimate the different summands separately.

\begin{remark}
To do our validated numerics we used SAGE \cite{SAGE} and the validated numerics packages shipped with it 
(the interval package is a binding to MPFI \cite{MPFI}), running either on local computers or on a cloud based version called 
Cocalc, https://cocalc.com/.

The discretized operators are computed using a rigorous interval Newton method \cite{Tucker}, 
while the estimates for the norm of the discretized operators are done using Scipy \cite{Scipy}, 
with rigorous error bounds on the error made by matrix-vector products, obtained through the implementation
of the rigorous matrix-vector product of \cite{Rump99}.

The experiment can be done on SageMathCloud following the SAGE worksheets contained in the software archive:
\begin{itemize}
\item \textit{stochastic\_estimate\_tail.sagews} bounds the size of the tail $||\sum_{i=l^*}^{+\infty}\PF^i\hat{\PF}h||_{\infty}$;
\item \textit{stochastic\_C\_1\_part.sagews} approximates the invariant density in the $C^1$ norm, to approximate $\hat{\PF}h$;
\item \textit{stochastic\_final\_estimate.sagews} estimates the error on the linear response and computes the approximation.
\end{itemize}

The software package contains a subset of the project compinvmeas-python, a software package designed to approximate invariant 
measures and associated objects.
There is a git repository for the full project whose access is by invitation: please send us an email so we can grant you access.

Part of the computation was done taking advantage of parallelization; to do so, some delicate memory issues arose \cite{2017-Nisoli-blog}.
\end{remark}

\subsection{Estimating $||\sum_{i=l^*}^{+\infty}\mathcal{L}^i\hat{\mathcal{L}}h||$}\label{subseq:contraction_stochastic}
Let $\eta=1/262144$ and let $\mathcal{L}_{\eta}$ be the discretized operator;
we have 
\begin{equation}\label{eq:decay_time_stochastic}
||\mathcal{L}_{\eta}^{9}|_{V_C^0}||_{\infty}\leq 9.39\cdot 10^{-5}
\end{equation}
where $V_C^0=\{f \in {C}^1, \int f=0 \}$ is the set zero average densities, and  for all $f\in {C}^1$, by the approximation  Lemma \ref{18}:
\[
||\mathcal{L}_{\eta}^{9}f-\mathcal{L}^{9}f||_{\infty}\leq 6.01\cdot 10^{-7}||f||_{{C}^1}+0.0021||f||_{\infty}.
\]
We can bound the speed of convergence to equilibrium and the associated constants 
using the technique and the notation explained in subsection \ref{sec:GNS}, with  $n_1=9$ .
For any $g$ in $V_{C^0}$, we have that, denoting by $g_i=\mathcal{L}^{7} g_{i-1}$:
\begin{equation*}
\left( 
\begin{array}{c}
||g_{i+1}||_{{C}^1} \\ 
||g_{i+1}||_{\infty}%
\end{array}%
\right) \preceq \left( 
\begin{array}{cc}
 1.0006\cdot 10^{-8} & 1.44 \\
 6.01\cdot 10^{-7} & 0.0022 
\end{array}
\right) \left( 
\begin{array}{c}
||g_{i}||_{{C}^1} \\ 
||g_{i}||_{\infty}%
\end{array}%
\right),
\end{equation*}
which gives us the following estimates
\[
||\mathcal{L}^{9k}f||_{\infty}\leq 1.00025 (0.0024)^k||f||_{{C}_1}\quad ||\mathcal{L}^{9k}f||_{\infty}\leq 4132.2 (0.0024)^k||f||_{{C}_1}.
\]
Therefore,
\[
||\sum_{i=18}^{\infty}\mathcal{L}^i\hat{\mathcal{L}}h||_{\infty}\leq 0.00037\gamma.
\]

\begin{remark}\label{rem:why_coarse}
Note that in this Subsection the partition size is coarser than the ones used in Subsection \ref{subsec:approximated_iteration} and \ref{subsec:approx_derivative};
the reason is that estimating \eqref{eq:decay_time_stochastic} is the most computationally expensive part of our algorithm.

Let $m=1/\eta$, the space of zero average measures of a partition of size $\eta$ has dimension $m-1$; the way we compute
\eqref{eq:decay_time_stochastic} rigorously is to choose a basis of the space of average zero measures and explicitly multiply
a sparse matrix with each element of this basis, which implies that the computation time scales asymptotically as $m^2$ 
(see \cite[Section 8.3]{GN} for a complete treatment in the $L^1$ case). Therefore, to speed up computations, it is worth computing \eqref{eq:decay_time_stochastic} on a coarser partition and
get information on $\PF$ by using \cite{GNS}.

Since the first submission of the article we developed more efficient tecniques to estimate these bounds, using what we call ``coarse-fine'' estimates \cite{GaMonNiPol}. 
Since the article presenting these results is still a work in progress, we decided
not to use them to do the computations in the present article, therefore the execution time of the examples is long but does not represent the state of the art
of our theory.
\end{remark}

\subsection{Estimating $||\sum_{i=0}^{l^*-1}\mathcal{L}^i f_{\eta}-\sum_{i=0}^{l^*-1}\mathcal{L}^i_{\eta}f_{\eta}||_{\infty}$}\label{subsec:approximated_iteration}
Let $\eta=1/524288$, we have that
\begin{equation*}
\begin{array}{cccc}
 N & ||\mathcal{L}^N_{\eta}f_{\eta}||_{{C}^1} & N & ||\mathcal{L}^N_{\eta}f_{\eta}||_{{C}^1}\\
 0 & 2.72 & 3 & 3.35\cdot 10^{-7} \\
 1 & 0.0007 & 4 & 1.07 \cdot 10^{-8} \\
 2 & 1.44\cdot 10^{-5} & 5 & 4.8\cdot 10^{-9}.
\end{array}
\end{equation*}
We observe now that
\[
||\mathcal{L}_{\eta}^k f||_{C^1}\leq 1.50002\cdot 0.261^k ||f||_{C^1}+1.42||f||_{\infty},
\]
which permits us to bound $||\mathcal{L}_{\eta}^k||_{C^1}$, and obtain that, for $6\leq N\leq 17$:
\[
||\mathcal{L}_{\eta}^N f_{\eta}||_{C^1} \leq 8.67\cdot 10^{-9}
\]

Therefore
\[
||\sum_{i=0}^{17}\mathcal{L}^i f_{\eta}-\sum_{i=0}^{17}\mathcal{L}^i_{\eta}f_{\eta}||_{\infty}\leq 0.0018 \gamma.
\]

\subsection{Estimating $||\sum_{i=0}^{l^*-1}\mathcal{L}^i(f_{\eta}-\hat{\mathcal{L}}h)||_{\infty}$}\label{subsec:approx_derivative}
Let $\eta=1/524288$, we computed using the discretization in Subsections \ref{subsec:approxLhatstochastic} and 7.5
an approximation $h_{\eta}$ of $h$ such that
\[
||h-h_{\eta}||_{{C}^1}\leq 0.00022.
\]
Therefore we have an approximation $f_{\eta}$ to $\hat{\mathcal{L}}h=\gamma h'$ such that
\[
||f_{\eta}-\gamma h'||_{\infty}\leq 0.00022\gamma.
\]
Therefore
\[
||\sum_{i=0}^{17}\mathcal{L}^i(f_{\eta}-\hat{\mathcal{L}}h)||_{\infty}\leq 0.0047\gamma.
\]
\subsection{The error on the computed response}
Therefore, the error on the response is
\[
||\hat{h}-\hat{h}_{appr}||_{\infty}\leq \gamma(0.00037+0.0018+0.0047)\leq 0.0069\gamma.
\]

\section{Linear response for deterministic perturbations}\label{sec:deterministic}
We now consider deterministic perturbations of a $C^3$ expanding circle map\footnote{After posting the first version of our work on Arxiv, which did not include an example of a deterministic perturbation,  Pollicott and Vytnova \cite{PV} studied the problem of approximating  linear response of given observables for deterministic perturbations. Their approach, which is based on the periodic point structure of the system, requires the maps to be analytic and the observables to have a certain structure. Here we show that our approach also works for deterministic perturbations, it requires only little regularity and no information on the observable.} $T_0$.  Let
\[T_{\epsilon}(x)=T_0(x)+\epsilon S(x)+ o_{C^3}(\epsilon);\]
where $S(x)\in C^3(\mathbb{T})$ and $o_{C^3}(\epsilon)$ is a term whose $C^3$ norm goes to $0$ faster than $\epsilon $ as $\epsilon$ goes to $0$. Under these assumptions (see for instance \cite{BahSau, Ba2, GaPo}), the operator
\[
\hat{\mathcal{L}}f(x)=-\PF_0\bigg(\frac{f\cdot S'}{T'_0}\bigg)(x)-\PF_0\bigg(\frac{f'\cdot S}{T'_0}\bigg)(x)+\PF_0\bigg(f\cdot S\cdot \frac{T''_0}{(T'_0)^2}\bigg)(x)
\]
satisfies
\begin{equation*}
\lim_{\epsilon \rightarrow 0}||\epsilon ^{-1}(\mathcal{L}-\mathcal{L}_{\epsilon })f-\mathcal{\hat{L}}f||_{C^{1}}=0~\forall f\in C^{3},
\end{equation*}
where $\mathcal{L}_{\epsilon}$ is the Perron-Frobenius operator associated to $T_{\epsilon}$.
\begin{remark}\label{deter}
If we suppose that the perturbation is small in the $C^2$ norm: $||T_0-T_\epsilon ||_{C^2}\leq K \epsilon$ it follows that
$\mathcal{L}_{\epsilon }$ and $\mathcal{L}$, satisfy a uniform Lasota-Yorke
inequality. This implies that assumption (1) of Proposition \ref{prop1}  holds.
Moreover, by  \cite{notes}, Section 6,  assumption (2) and  (3) of Proposition \ref{prop1} also hold.
Hence Proposition \ref{prop1} holds and we have the linear response for these perturbations.
\begin{equation*}
\lim_{\epsilon \rightarrow 0}||\frac{h_{\epsilon }-h}{\epsilon }-%
\hat{h}||_{\infty }=0,  
\end{equation*}
where $\hat{h}:=(\Id-\mathcal{L})^{-1}\mathcal{\hat{L}}h.$
\end{remark}

\subsection{An example of linear response under a deterministic perturbation}
In this example we study a family  $T_{\epsilon}$, $\epsilon \in [0,1)$ of  $C^3$-small deterministic  perturbations. We consider the family $T_{\epsilon}$:
\[
T_{\eps}(x)=2x+\frac{\epsilon }{16}(\cos(4\pi x)+\frac{1}{4}\cos(8\pi x)) \textrm{ mod $1$}.
\]

For $\epsilon=0$  the dynamics is given by the map
\[
T_0(x)=2x \textrm{ mod $1$};
\]
whose invariant density $h$ is constant and equals to $1$. 
The operator $\mathcal{L}_0$, associated to $T_0$, satisfies the following inequalities:
\begin{align*}
&\textrm{Var}(\mathcal{L}_0^k f)\leq \bigg(\frac{1}{2}\bigg)^k\textrm{Var}(f)\\
&M=||\mathcal{L}_0||_{\infty}\leq 1\\
&||\mathcal{L}_0^k f||_{C^1}\leq \bigg(\frac{1}{2}\bigg)^k ||f||_{C^1}+||f||_{\infty}.
\end{align*}
Note that the family satisfies the assumptions discussed in Remark \ref{deter}.
 Hence the linear response formula can be applied.
  Following \cite{GaPo}, the operator
$\hat{\mathcal{L}}:C^2\to C^1$ is given by
\begin{equation}\label{llll}
\hat{\mathcal{L}}h=\mathcal{L}_0\bigg(\frac{\pi}{8}\sin(4\pi x)+\frac{\pi}{16} \sin(8\pi x) \bigg)=\frac{\pi}{8}\sin(2\pi x)+\frac{\pi}{16}\sin(4\pi x)
\end{equation}
and, by Proposition \ref{prop1}, the linear response is given by
\[
\hat{h}=(\Id-\mathcal{L}_0)^{-1}\hat{\mathcal{L}}h=\sum_{i=0}^{+\infty}\mathcal{L}^i_0 \hat{\mathcal{L}}h.
\]

The simple structure of the example also  allow to compute the response exactly.

\begin{remark}\label{exact}
From direct computations we have that for all $\phi\in L^1$
\[
\mathcal{L}_0\phi(x)=\frac{\phi(\frac{x}{2})+\phi(\frac{x}{2}+\frac{1}{2})}{2}.
\]
Therefore, $\mathcal{L}_0\sin(2\pi x)=0$, $\mathcal{L}_0\sin(4\pi x)=\sin(2\pi x)$ and 
\[
\hat{h}=\sum_{i=0}^{+\infty}\mathcal{L}^i_0 \hat{\mathcal{L}}h=\hat{\mathcal{L}}h+\mathcal{L}_0\hat{\mathcal{L}}h=\frac{1}{16}\left(3\pi\sin(2\pi x)+\pi\sin(4\pi x)\right).
\]
\end{remark} 
We now approximate the linear response and estimate the error using the algorithm developed in the present paper
(refer to Theorem \ref{main} and the subsequent discussion). 
We will compute its linear response using the discretized operator and the general estimates introduced in Section \ref{noise}. 
Let us set the discretization parameter $\eta=1/4194304$. We have:
\begin{align*}
||\hat{\mathcal{L}}h||_{\infty}&\leq \frac{3\pi}{16};\\
||\hat{\mathcal{L}}h||_{C^1}&\leq \frac{3\pi}{16} +\frac{\pi^2}{2}.
 \end{align*}
As $\hat{\mathcal{L}}h$ is explicit (see Equation \ref{llll}) in the algorithm we use its discretization
$f_{\eta}=\Pi_{\eta}\hat{\mathcal{L}}h$.  Let us choose $l^*=57$ and approximate $\hat{h}$ by $$\hat{h}_{appr}=\sum_{i=0}^{56}\mathcal{L}_{\eta}f_{\eta}.$$ 
In Figure \ref{fig:linear_response_prescribed} we have a plot of the approximated linear response $\hat{h}_{appr}$.  

\bigskip

Now we apply the general procedure for the estimation of the error. 
\begin{figure}
 \includegraphics[width=50mm,height=50mm]{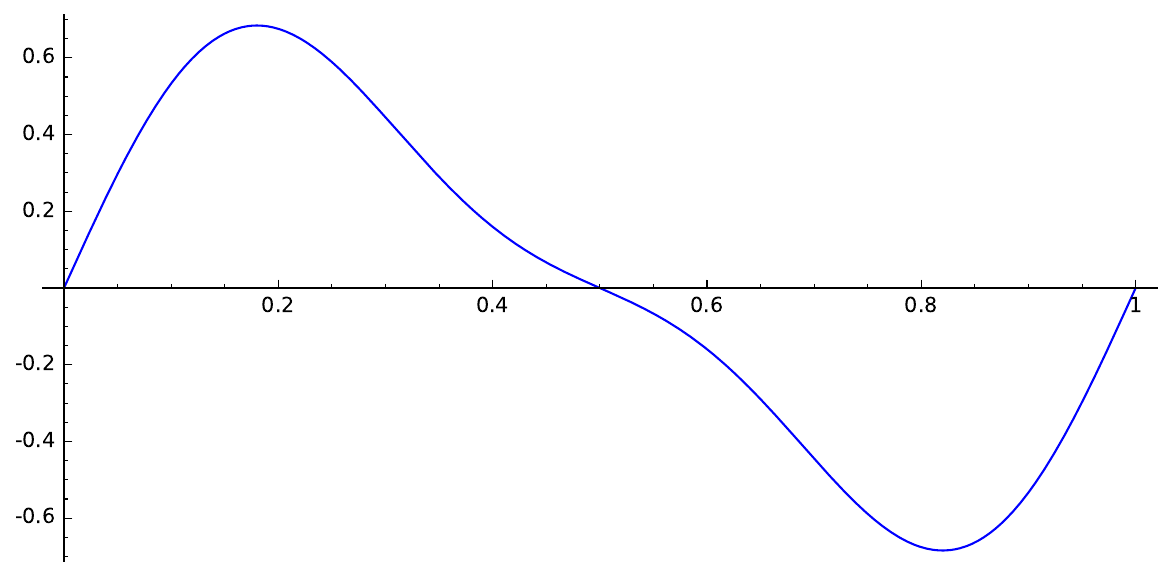}
 \caption{The plot of $\hat{h}_{appr}=\sum_{i=0}^{56}\mathcal{L}_{\eta}f_{\eta}$. }
 \label{fig:linear_response_prescribed}
\end{figure} 
As in the previous section, we have to estimate three summands:
\begin{align*}
&||\hat{h}-\hat{h}_{appr}||_{\infty}\leq \\
&||\sum_{i=l^*}^{+\infty}\mathcal{L}^i\hat{\mathcal{L}}h||_{\infty}+||\sum_{i=0}^{l^*-1}\mathcal{L}^i f_{\eta}-\sum_{i=0}^{l^*-1}\mathcal{L}^i_{\eta}f_{\eta}||_{\infty}+||\sum_{i=0}^{l^*-1}\mathcal{L}^i (f_{\eta}-\hat{\mathcal{L}}h)||_{\infty}.
\end{align*}
In the following  we are going to estimate the different summands separately.

\begin{remark}
The experiment can be done on SageMathCloud following the SAGE worksheets contained in the software archive:
\begin{itemize}
\item \textit{deterministic\_estimate\_tail.sagews} bounds the size of the tail $||\sum_{i=l^*}^{+\infty}\PF^i\hat{\PF}h||_{\infty}$;
\item \textit{deterministic\_final\_estimate.sagews} estimates the error on the linear response and computes the approximation.
\end{itemize}
\end{remark}

\subsection{Estimating $||\sum_{i=l^*}^{+\infty}\mathcal{L}^i\hat{\mathcal{L}}h||$}
Let us consider a coarse discretization $\eta=1/131072$ (see Remark \ref{rem:why_coarse}) 
let $\mathcal{L}_{\eta}$ be its discretized operator.  We have  by direct computation
\[
||\mathcal{L}_{\eta}^{19}|_{V_{C}^0}||_{\infty}\leq 0.00076
\]
where $V_C^0)=\{f \in {C}^1, \int f=0 \}$ is the set zero average densities. For all $f\in {C}^1$, by Lemma \ref{18}:
\[
||\mathcal{L}_{\eta}^{19}f-\mathcal{L}^{19}f||_{\infty}\leq 0.000144||f||_{{C}^1}+0.0079||f||_{\infty}.
\]
We can bound the speed of convergence to equilibrium and the associated constants 
using the technique and the notation explained in subsection \ref{sec:GNS}, with  $n_1=19$ .
For any $g$ in $V_{C^0}$, we have that, denoting by $g_i=\mathcal{L}^{19} g_{i-1}$:
\begin{equation*}
\left( 
\begin{array}{c}
||g_{i+1}||_{{C}^1} \\ 
||g_{i+1}||_{\infty}%
\end{array}%
\right) \preceq \left( 
\begin{array}{cc}
 1.91\cdot 10^{-6} & 1 \\
 0.000144 & 0.0079 
\end{array}
\right) \left( 
\begin{array}{c}
||g_{i}||_{{C}^1} \\ 
||g_{i}||_{\infty}%
\end{array}%
\right).
\end{equation*}
This gives the following estimates for all $f \in V$ 
\[
||\mathcal{L}^{19k}f||_{\infty}\leq 1.01 (0.0171)^k||f||_{{C}_1}\quad ||\mathcal{L}^{19k}f||_{\infty}\leq 120.1 (0.0171)^k||f||_{{C}_1}.
\]
Therefore,
\[
||\sum_{i=57}^{\infty}\mathcal{L}^i\hat{\mathcal{L}}h||_{\infty}\leq 0.00055.
\]
\subsection{Estimating $||\sum_{i=0}^{l^*-1}\mathcal{L}^i f_{\eta}-\sum_{i=0}^{l^*-1}\mathcal{L}^i_{\eta}f_{\eta}||_{\infty}$}
Let $\eta=1/4194304$, we have that
\begin{equation*}
\begin{array}{cccc}
 N & ||\mathcal{L}^N_{\eta}f_{\eta}||_{{C}^1} & N & ||\mathcal{L}^N_{\eta}f_{\eta}||_{{C}^1}\\
 0 & 7.92 & 2 & 6.42\cdot 10^{-10}\\
 1 & 2.047 & 3 & 3.24\cdot 10^{-10}.
\end{array}
\end{equation*}
We observe now that
\[
||\mathcal{L}_{\eta}^k f||_{C^1}\leq 1.6\cdot 0.76^k ||f||_{C^1}+1||f||_{\infty},
\]
which permits us to bound $||\mathcal{L}_{\eta}^k||_{C^1}$, and obtain that, for $4\leq N\leq 56$:
\[
||\mathcal{L}_{\eta}^N f_{\eta}||_{C^1} \leq 3.31\cdot 10^{-10}
\]
Therefore,
\[
||\sum_{i=0}^{56}\mathcal{L}^i f_{\eta}-\sum_{i=0}^{56}\mathcal{L}^i_{\eta}f_{\eta}||_{\infty}\leq  0.0019
\]
\subsection{Estimating $||\sum_{i=0}^{l^*-1}\mathcal{L}^i(f_{\eta}-\hat{\mathcal{L}}h)||_{\infty}$}
Let $\eta=1/4194304$, we have that
\[
||f_{\eta}-\hat{\mathcal{L}}h||_{\infty}=||\Pi_{\eta}(\hat{\mathcal{L}}h)-\hat{\mathcal{L}}h||_{\infty}\leq \frac{5}{2}\frac{1}{4194304}\frac{\pi^2}{2}\leq 2.96\cdot 10^{-6}.
\]
Therefore,
\[
||\sum_{i=0}^{56}\mathcal{L}^i(f_{\eta}-\hat{\mathcal{L}}h)||_{\infty}\leq 0.00017.
\]
\subsection{The error on the computed response}
Therefore, 
\[
||\hat{h}-\hat{h}_{appr}||_{\infty}\leq 0.00055+0.0019+0.00017\leq 0.0026.
\]
\begin{remark}
We expect our approximation scheme to be more efficient than the rigorous bounds.
Since we computed $\hat{h}$ explicitly in Remark \ref{exact}, we compare
our approximation directly with the explicit linear response obtaining:
\[
||\hat h_{appr}-\Pi_{\eta}\hat{h}||_{C^0}\leq 3.42\cdot 10^{-13},\quad ||\Pi_{\eta}\hat{h}-\hat{h}||_{C^0}\leq \frac{15.5}{4194304}.
\]
This shows the efficiency of our approximation scheme.
\end{remark}
\section{Appendix I: the response of the intermittent family at the boundary}\label{appendix I} 
Let $T_{\alpha}:[0,1]\to [0,1]$ be the family of maps defined by
\begin{equation}\label{LSV}
T_{\alpha}(x)=\left\{
\begin{array}{cc}
x(1+2^{\alpha}x^{\alpha}) & 0\leq x\leq 1/2\\
2x-1 & 1/2< x\leq 1,
\end{array}
\right.
\end{equation}
where $\alpha\in[0,\infty)$. This family which was initially 
popularized in \cite{LSV} as a version of the Pomeau-Manneville family \cite{PM}. 
When $\alpha\in(0,\infty)$, this family of maps of the interval $[0,1]$ has an indifferent fixed point at the origin, 
and when $\alpha\in (0,1)$ it has a unique absolutely continuous invariant probability measure and it exhibits only a 
polynomial decay of correlations \cite{G} with respect to H\"older observables. When $\alpha=0$, it is the doubling map, 
which is uniformly expanding, it preserves Lebesgue measure and it exhibits exponential decay of correlations. In this appendix,
we compute the linear response at $\alpha=0$.\\

In general for $\alpha\in[0,\infty)$, the transfer operator associated with $T_\alpha$ has a unique (up to multiplication by a constant) fixed point $h_\alpha\in L^1([0,1])$. The following result was proved in \cite{BT}. In particular, the linear response formula at $\alpha=0$ in Proposition \ref{atz} below was derived in equation (2.22) of \cite{BT}.
\begin{proposition}\label{atz}
Let $T_{\alpha}$ be as in \eqref{LSV} with $\alpha\in [0,1)$.
\begin{enumerate}
\item For $\alpha\in(0,1)$, $q> \frac{1}{1-\alpha}$ and $\psi\in L^q$
\begin{equation}\label{BTR}
\lim_{\eps\to 0}\bigg|\frac{\int \psi h_{\alpha+\eps}dx-\int \psi h_{\alpha}dx}{\eps}-\int\psi \hat h_{\alpha}dx\bigg|=0,
\end{equation}
where $\hat h_{\alpha}=-(\Id- \mathcal L_{\alpha})^{-1}[(X_\alpha\mathcal N_\alpha h_{\alpha})']$, with 
$X_\alpha=2^{\alpha}g_{0,\alpha}^{1+\alpha}(x)\ln(2g_{0,\alpha}(x))$, $\mathcal N_\alpha \varphi=g'_{0,\alpha}(x)\varphi(g_{0,\alpha}(x))$.
\item The result also holds for  $\alpha=0$ by taking the limit as $\eps\downarrow 0$. The formula of $\hat h:=\hat h_0$ is given by
\begin{equation}\label{eq:lin_resp_formula_for_LSV}
\hat{h}=(\Id-\mathcal{L}_0)^{-1}\hat{\mathcal{L}}h_0, 
\end{equation}
where
\begin{equation}\label{it:hatL}
\hat{\mathcal{L}}h_0=-1/4-\ln(x)/4 \text{ and } h_0\equiv 1.
\end{equation}
\end{enumerate}
\end{proposition}
We now start the procedure of approximating $\hat h$ rigorously in $L^1([0,1])$. In this specific case, it is possible to find some explicit bounds that allow us to approximate the linear response. Note that approximating $\hat h$ in $L^1([0,1])$ will directly provide an approximation of $\int\psi \hat h dx$, for any $\psi\in L^{\infty}$. Since some of the techniques differ from those used in the rest of the paper, we decided to present this example as an appendix. Note that for $T_0(x)=2x \textrm{ mod $1$}$; the transfer operator on $L^1([0,1])$ associated to this dynamical
system has the explicit form
\[
\mathcal{L}_0 f(x)=\frac{f(x/2)+f(x/2+1/2)}{2},
\]
and the density of the absolutely continuous invariant measure is $h_0\equiv 1$;
those are important ingredients in our estimates.

\begin{remark}
The experiment can be done can be done on Cocalc following the SAGE worksheets contained in the software archive:
\textit{lsv\_at\_boundary.sagews} estimates the error on the linear response and computes the approximation.
\end{remark}

\begin{definition}
Let $\{I_i\}_{i=0}^{n-1}$ be a uniform partition of $[0,1]$ consisting of intervals of size $\eta=1/n$, denote
by $m$ the Lebesgue measure. Let $\pi_{\eta}$ be the finite rank operator defined on $L^1([0,1])$ as follows:
\[
\pi_{\eta}f(x)=\sum_{i=0}^{n-1}\frac{\int_{I_i}f\,dm}{m(I_i)}\chi_{I_i}(x),
\]
where $\chi_{I_i}$ is the characteristic function of $I_i$. The Ulam approximation of $\mathcal{L}_0$ of mesh size $\eta$ is
\[
\mathcal{L}_{0,\eta}:=\pi_{\eta}\mathcal{L}\pi_{\eta}.
\]
\end{definition}
We summarize in the next lemma some properties of $\pi_{\eta}$,
$\mathcal{L}_0$ and $\mathcal{L}_{0,\eta}$ used in this appendix;
we refer to \cite{B1,GN} and references therein for proofs of
these results.

\begin{lemma}\label{lemma:pf}
Let $T_0(x)=2x \textrm{ mod $1$}$, let $\mathcal{L}_0$ be the associated transfer operator,
and let $\mathcal{L}_{0,\eta}$ be the Ulam approximation of size $\eta$.
Then
\begin{align*}
||g-\pi_{\eta}g||_{L^1}&\leq \eta\cdot \Var(g)\\
||\PF_0||_{L^1}\leq 1 &\quad ||\PF_{0,\eta}||_{L^1}\leq 1\\
 \Var(\PF_0 g)\leq \frac{1}{2}\Var(g) &\quad\Var(\PF_{0,\eta} g)\leq \frac{1}{2}\Var(g),
\end{align*}
for all $g\in BV([0,1])$.
\end{lemma}

We now approximate $\hat h$, which was defined in equations \eqref{eq:lin_resp_formula_for_LSV} and \eqref{it:hatL}, with a rigorous error bound in $L^1([0,1])$. 
\begin{proposition}\label{prop:summary}
Let $\mathcal{L}_{0,\eta}$ 
be the Ulam approximation of $\mathcal{L}_0$ with mesh size $\eta$.
Let $\eta=2^{-20}$, $f_{\eta}=\pi_{\eta}\hat{\mathcal{L}} h_0$ and $\hat{h}_{appr}=\sum_{i=0}^{19}\mathcal{L}_{\eta}f_{\eta}$. Then,
\[
||\hat{h}_{appr}-\hat{h}||_{L^1}\leq 0.00021.
\]
\end{proposition}
\begin{proof}
To obtain the rigorous bound for $||\hat{h}_{appr}-\hat{h}||_{L^1}$, note that
\begin{align*}
||\hat{h}_{appr}-\hat{h}||_{L^1}&=||\sum_{i=0}^{19}\PF_{0,\eta}^i f_{\eta}-(\Id-\PF_0)^{-1}\hat{\PF}h_0||_{L^1}\leq \\
||\sum_{i=0}^{19}\mathcal{L}_0^i(f_{\eta}-&\hat{\mathcal{L}}h_0)||_{L^1}+||\sum_{i=0}^{19}\mathcal{L}_0^i f_{\eta}-\sum_{i=0}^{19}\mathcal{L}^i_{0,\eta}f_{\eta}||_{L^1}+||\sum_{i=20}^{+\infty}\mathcal{L}^i_0\hat{\mathcal{L}}h_0||_{L^1}.
\end{align*}

We will now give an explicit estimate for $||\mathcal{L}^k_0\hat{\mathcal{L}}h_0||_{L^1}$. By induction we have 
\begin{align*}
|\mathcal{L}^k_0\hat{\mathcal{L}}h_0(x)|\leq |\bigg(-\frac{1}{4}-\frac{1}{4}\cdot \bigg(\frac{\ln(\prod_{j=1}^{2^k-1} (x+j))}{2^k}-\ln(2^k)\bigg)\bigg)|+|\frac{\ln(x)}{2^{k+2}}|.
\end{align*}
Note that if $x\in[0,1]$, we have that:
\[
(2^k-1)! \leq \prod_{j=1}^{2^k-1} (x+j)\leq 2^k!;
\]
we can use Stirling's Formula for $n!$,  \cite{Robbins}:
\[
\ln(\sqrt{2\pi})+(n+\frac{1}{2})\ln(n)-n+\frac{1}{12n+1}\leq \ln(n!)\leq \ln(\sqrt{2\pi})+(n+\frac{1}{2})\ln(n)-n+\frac{1}{12n}.
\]
Thus:
\begin{align*}
\bigg|1+\frac{\ln(\prod_{j=1}^{2^k-1} (x+j))}{2^k}-\ln(2^k)\bigg|\leq
\frac{\ln(\sqrt{2\pi})}{2^k}+\frac{\ln(2^k)}{2^{k+1}}+\frac{1}{2^{2k}12}, 
\end{align*}
which in turn implies that
\begin{align*}
||\mathcal{L}^k_0\hat{\mathcal{L}}h_0(x)||_{L^1}\leq \frac{1}{2^{k+2}}\cdot \bigg(1+\ln(\sqrt{2\pi})+\frac{\ln(2^k)}{2}+\frac{1}{12\cdot 2^k}\bigg).
\end{align*}
Therefore,
\begin{equation}\label{eq:tail}
||\sum_{i=20}^{+\infty}\mathcal{L}^i_0\hat{\mathcal{L}}h_0||_{L^1}\leq \frac{1}{2^{20+1}}(1+\ln(\sqrt{2\pi}))+\frac{20+1}{2^{20+1}}\frac{\ln(2)}{2}+\frac{1}{9\cdot 2^{40+2}},
\end{equation}
and consequently,
\[||\sum_{i=20}^{+\infty}\mathcal{L}^i_0\hat{\mathcal{L}}h_0||_{L^1}\leq 4.4\cdot 10^{-6}.\]

We estimate $||f_{\eta}-\hat{\mathcal{L}}h_0||_{L^1}$ explicitly:
\begin{align*}
\int_0^{\eta}|\ln(x)-\frac{\int_0^{\eta} \ln(x)dx}{\eta}|dx&=\frac{2\eta}{e};
\end{align*}
for each interval $[i\eta,(i+1)\eta]$, $i\in 1,\ldots,1/\eta-1$ we have that
\begin{align*}
\int_{i\eta}^{(i+1)\eta}|\ln(x)-\frac{1}{\eta}\int_{i\eta}^{(i+1)\eta}\ln(\xi)d\xi|dx\leq \eta(\ln(i\eta)-\ln((i+1)\eta)).
\end{align*}
Therefore:
\[
||f_{\eta}-\hat{\mathcal{L}}h_0||_{L^1}\leq \frac{\eta}{2e}-\frac{\eta\ln(\eta)}{4}.
\]
Since $||\mathcal{L}_0||_{L^1}=1$ we have 
\begin{equation}\label{eq:othtail}
||\sum_{i=0}^{19}\mathcal{L}_0^i(f_{\eta}-\hat{\mathcal{L}}h_0)||_{L^1}\leq 20 \bigg(\frac{\eta}{2e}-\frac{\eta\ln(\eta)}{4}\bigg).
\end{equation}
Fix $\eta=2^{-20}$; therefore using \eqref{eq:othtail} we have:
\[
||\sum_{i=0}^{19}\mathcal{L}_0^i(f_{\eta}-\hat{\mathcal{L}}h_0)||_{L^1}\leq 6.9\cdot{10^{-6}}
\]

We bound now $||\sum_{i=0}^{19}\mathcal{L}_0^i f_{\eta}-\sum_{i=0}^{19}\mathcal{L}^i_{0,\eta}f_{\eta}||_{L^1}$.
We make and ``a priori'' estimate, using the fact that $\eta=1/n$ and $\ln(x)$ is decreasing:
\begin{align*}
\Var(f_{\eta})=\frac{|\int_0^{\eta}\hat{\mathcal{L}}h_0\, dx-\int_{(n-1)\eta}^1\hat{\mathcal{L}}h_0\, dx|}{\eta}=\frac{-\ln(\eta)-(n-1)\ln((n-1)\eta)}{4}
\end{align*}
From Lemma \ref{lemma:pf}
\[
||(\mathcal{L}_0-\mathcal{L}_{0,\eta})f||_{L^1}\leq \frac{3\eta}{2}\Var(f),
\]
and
\begin{align}
||\sum_{i=0}^{19}\mathcal{L}_0^i f_{\eta}-\sum_{i=0}^{19}\mathcal{L}^i_{0,\eta}f_{\eta}||_{L^1}
&\leq \sum_{j=0}^{19} \sum_{i=0}^j\frac{3\eta}{2}\bigg(\frac{1}{2}\bigg)^i\Var(f_{\eta})\nonumber \\
&\leq \frac{3\eta}{2} \bigg(2\cdot 20- \frac{2^{20}-1}{2^{19}}\bigg)\Var(f_{\eta}).\label{eq:diffiterate}
\end{align}
Using \eqref{eq:diffiterate}, since for $\eta=2^{-20}$ $\Var(f_{\eta})\leq 3.72$, we have the following bound:
\[
||\sum_{i=0}^{19}\mathcal{L}_0^i f_{\eta}-\sum_{i=0}^{19}\mathcal{L}^i_{0,\eta}f_{\eta}||_{L^1}\leq 0.000203.
\]

Summing up the errors we obtain:
\[
||\hat{h}_{appr}-\hat{h}||_{L^1}\leq 0.00021.
\]
\end{proof}
In Figure \ref{fig:linear_response_manneville}, we depict the graph of the computed approximation of the linear response.
\begin{figure}
 \includegraphics[width=50mm,height=50mm]{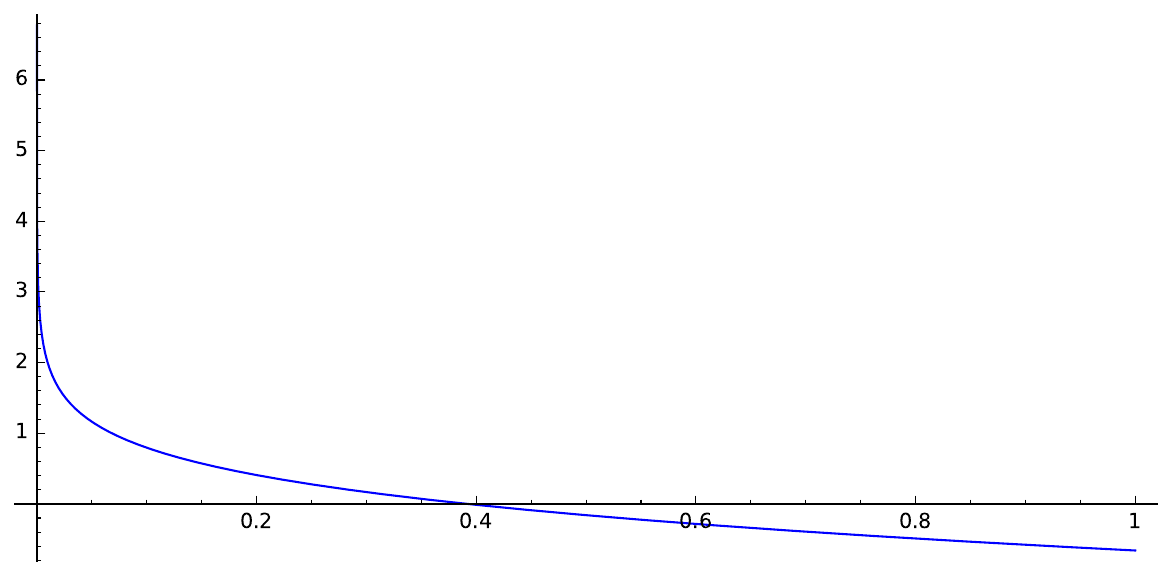}
 \caption{The plot of $\hat{h}_{appr}=\sum_{i=0}^{19}\mathcal{L}_{\eta}f_{\eta}$.}
 \label{fig:linear_response_manneville}
\end{figure}
\section{Appendix II: some estimates and technical lemmas}\label{appendix II}
Throughout subsections 7.1-7.5 we use the following setup. $T:\mathbb{T}\rightarrow \mathbb{T}$ is a $C^{3}$ uniformly expanding
circle map; i.e. $\inf_{x\in \mathbb{T}}|D_{x}T|>1$. Without loss of generality we assume that $T$ is orientation preserving. The circle map $T$ it is naturally associated with an 
expanding interval map, which we also denote by $T$, $T:[0,1]\to [0,1]$. Throughout the presentation, we use the interval map representation $T:[0,1]\to [0,1]$. Recall that $\mathcal L$ denotes the transfer operator associated with $T$ (see \eqref{PF}). In addition, the following constants will be used extensively throughout subsections 7.1-7.5. We set
$$\lambda :=1/\inf_{x\in \mathbb{T}}|D_{x}T|; \quad B:=||T^{\prime \prime }/(T^{\prime })^{2}||_{\infty };\quad M:=1+\frac{B}{1-\lambda};$$ 
\[
Z:= \frac{1}{1-\lambda^2}\bigg(\bigg|\bigg| \frac{T'''}{(T')^3}\bigg|\bigg|_{\infty}+
\frac{3\lambda}{1-\lambda}\bigg|\bigg| \frac{T''}{(T')^2}\bigg|\bigg|_{\infty}^2\bigg);
\]
and
\begin{equation*}
D=\max\{3\frac{\lambda BM}{1-\lambda},3M\bigg(\frac{B}{1-\lambda}\bigg)^2+MZ\}+M\lambda+M^2.
\end{equation*}
Finally, we denote the $k^{\text{th}}$ iterate of $T$ by $G_k$; i.e. for $k\ge 1$ we write $G_k:=T^k$.
\subsection{Useful estimates}
The following Lemma provides bounds on the distortion for iterates of $T$.
These bounds will be used in the proofs of the Lasota-Yorke inequalities in Subsection \ref{sec:LY}. 

\begin{lemma}\label{lemma:couple_estimates}
For any $k\ge 1$, we have 
\[
\bigg|\bigg| \frac{G_k''}{(G_k')^2}\bigg|\bigg|_{\infty}\leq \frac{B}{1-\lambda};\quad\text{ and }\quad \bigg|\bigg| \frac{G_k'''}{(G_k')^3}\bigg|\bigg|_{\infty}\leq Z.
\]
\end{lemma}
\begin{proof}
Write $G_k(x)=T(G_{k-1}(x))$. Then
\begin{align*}
 G_k'(x)&=T'(G_{k-1}(x))G_{k-1}'(x),\\
 G_k''(x)&=T''(G_{k-1}(x))(G_{k-1}'(x))^2+T'(G_{k-1}(x))G_{k-1}''(x).
\end{align*}
Using these two expressions we have
\begin{align*}
\frac{G''_k(x)}{(G_k'(x))^2}=\frac{T''(G_{k-1}(x))}{(T'(G_{k-1}(x)))^2}+\frac{1}{T'(G_{k-1}(x))}\frac{G_{k-1}''(x)}{(G_{k-1}'(x))^2},
\end{align*}
which implies the first inequality. We now compute
\begin{align*}
 G_k'''(x)&=T'''(G_{k-1}(x))(G_{k-1}'(x))^3+3 T''(G_{k-1}(x))G_{k-1}'(x)G_{k-1}''(x)\\&+T'(G_{k-1}(x))G_{k-1}'''(x).
\end{align*}
Using this last expression and the computations above we have:
\begin{align*}
\frac{G_k'''(x)}{(G_k'(x))^3}&=\frac{T'''(G_{k-1}(x))}{(T'(G_{k-1}(x)))^3}+3\frac{1}{T'(G_{k-1}(x))}\frac{T''(G_{k-1}(x))}{(T'(G_{k-1}(x)))^2}\frac{G_{k-1}''(x)}{(G_{k-1}'(x))^2}\\
&+\frac{1}{(T'(G_{k-1}(x)))^2}\frac{G_{k-1}'''(x)}{(G_{k-1}'(x))^3},
\end{align*}
which implies the second inequality of the lemma.
\end{proof}
\subsection{Lasota-Yorke inequalities\label{sec:LY}}
In this subsection we prove Lasota-Yorke inequalities when $\mathcal{L}$ acts on $C^1([0,1])$ and on $C^2([0,1])$. The following proposition is a well known result. See for instance \cite{Li2} Lemma 1.2 for a similar statement. 
\begin{proposition}
\label{LYvar} Let $\text{Var}(\cdot)$ denote the one dimensional variation on $[0,1]$. Then for any function of bounded variation $f$ we have 
\begin{equation*}
\text{Var}(\mathcal{L}f )\leq \lambda \text{Var}(f )+B||f||_{L^{1}}.
\end{equation*}
\end{proposition}
\begin{lemma}[Uniform bound on $||\mathcal{L}^{n}||_{\infty }$]
\label{M}For any $n\ge 1$ we have 
\begin{equation*}
||\mathcal{L}^{n}||_{\infty }\leq M.
\end{equation*}
\end{lemma}
\begin{proof}
The operator $\mathcal{L}$ is positive. Therefore $||\mathcal{L}%
^{n}||_{\infty }=\sup_{x\in [0,1]}\mathcal{L}^{n}1.$ By Proposition \ref%
{LYvar}, we have%
\begin{equation*}
\text{Var}(\mathcal{L}^{n}1)\leq \frac{B}{1-\lambda },
\end{equation*}%
and therefore $||\mathcal{L}^{n}1||_{\infty }\leq M$.
\end{proof}

\begin{proposition}
\label{Th:SwLY} For $f\in C^1([0,1])$ and any $n\ge 1$ we have
\begin{equation*}
||\mathcal{L}^n f||_{{C}^{1}}\leq M\cdot \lambda^n ||f||_{{C}%
^{1}}+M^2||f||_{\infty }.
\end{equation*}%
In particular, there exists an iterate $G:=T^n$ of $T$ such that
\[
||\mathcal{L}_G f||_{{C}^{1}}\leq \theta ||f||_{{C}%
^{1}}+M^2||f||_{\infty },
\]
where $\theta\leq\lambda^n M<1$.
\end{proposition}
\begin{proof}
For $x\in (0,1)$ we have 
\begin{equation}
\frac{\partial }{\partial x}(\mathcal{L}^n f)(x)=\frac{\partial }{\partial x}%
\bigg(\sum_{y\in G^{-1}(x)}\frac{f(y)}{G^{\prime }(y)}\bigg)=\sum_{y\in G^{-1}(x)}%
\frac{f^{\prime }(y)}{(G^{\prime }y)^{2}}-f(y)\frac{G^{\prime \prime }(y)}{%
(G^{\prime }y)^{2}}\frac{1}{G^{\prime }(y)}.  \label{former}
\end{equation}%
By Lemma \ref{M}, Lemma \ref{lemma:couple_estimates} and (\ref{former}) 
\begin{equation}
||(\mathcal{L}^nf)^{\prime }||_{\infty }\leq \lambda^n ||\mathcal{L}^nf^{\prime
}||_{\infty }+ \bigg|\bigg|\frac{G^{\prime \prime }}{(G^{\prime })^{2}}%
\bigg|\bigg|_{\infty }||\mathcal{L}^n f||_{\infty }
\leq M \lambda^n ||f^{\prime }||_{\infty }+\frac{BM}{1-\lambda}||f||_{\infty }%
.  \label{me1}
\end{equation}%
Therefore, by \eqref{me1} and Lemma \ref{M}, we have 
\begin{equation*}
\begin{split}
||\mathcal{L}^n f||_{{C}^{1}}& =||\mathcal{L}^n f||_{\infty }+||(\mathcal{L}^n f)^{\prime }||_{\infty } \\
& \leq M||f||_{\infty }+M \lambda^n ||f^{\prime }||_{\infty
}+\frac{BM}{1-\lambda}||f||_{\infty } \\
& \leq M \lambda^n ||f||_{{C}^{1}}+M^2||f||_{\infty }.
\end{split}%
\end{equation*}
\end{proof}

\begin{proposition}\label{prop:C1L1}
For $f\in C^1$, we have
\begin{equation*}
||(\mathcal{L} f)'||_{\infty}\leq (M+B)\lambda||f'||_{\infty}+B(B+1)||f||_{L^1},
\end{equation*}
\end{proposition}
\begin{proof}
From Proposition \ref{LYvar} and the fact that $||f||_{\infty}\leq ||f||_{L^1}+\Var(f)$
we have
\begin{align*}
||\PF f||_{\infty}\leq \Var(\PF f)+||\PF f||_{L^1}\leq \lambda \textrm{Var}(f)+(B+1)||f||_{L^1}\\
\leq \lambda ||f'||_{\infty}+(B+1)||f||_{L^1}.
\end{align*}
By \eqref{me1} and the above estimate we have
\begin{align*}
||(\PF f)'||_{\infty}\leq M\lambda||f'||_{\infty}+B ||\PF f||_{\infty}\leq (M+B)\lambda ||f'||_{\infty}+B(B+1) ||f||_{L^1}.
\end{align*}
\end{proof}
\begin{proposition}
\label{Th:SSLY} For $f\in C^2([0,1])$ and any $n\ge 1$ we have
\begin{equation*}
||\mathcal{L}^n f||_{{C}^{2}}\leq M (\lambda^2)^n ||f||_{{C}
^{2}}+D||f||_{{C}^{1}}.
\end{equation*}
In particular, there exists an iterate $G:=T^k$ of $T$ such that
\[
||\mathcal{L}_G f||_{{C}^{2}}\leq \Lambda ||f||_{{C}^{2}}+D||f||_{{C}^{2}},
\]
where $\Lambda\leq\lambda^{2k} M<1$.
\end{proposition}
\begin{proof}
We denote $G:=T^n$. For $x\in (0,1)$ we have 
\begin{align*}
(\mathcal{L}^n f)^{\prime \prime }(x)=& \sum_{y\in G^{-1}(x)}\frac{f^{\prime
\prime }(y)}{(G^{\prime }(y))^{3}}-3f^{\prime }(y)\frac{G^{\prime \prime }(y)%
}{(G^{\prime }(y))^{4}} \\
& +\sum_{y\in G^{-1}(x)}-f(y)\frac{G^{\prime \prime \prime }(y)}{(G^{\prime
}(y))^{4}}+3f(y)\frac{(G^{\prime \prime }(y))^{2}}{(G^{\prime }(y))^{5}}.
\end{align*}
Therefore, 
\begin{align}\label{eq:really_useful}
& ||(\mathcal{L}^n f)^{\prime \prime }||_{\infty }\leq \lambda ^{2n}||\mathcal{L}^n%
(f^{\prime \prime })||_{\infty }+3\lambda^n \frac{B}{1-\lambda}||\mathcal{L}^n(f^{\prime
})||_{\infty } \\ \nonumber
& +3\bigg(\frac{B}{1-\lambda}\bigg)^{2}||\mathcal{L}^n f||_{\infty }+Z ||\mathcal{L}^n f||_{\infty }.
\end{align}%
In particular
\[
||(\mathcal{L}^n f)^{\prime \prime }||_{\infty }\leq M\lambda ^{2n}||f^{\prime \prime }||_{\infty}
+\max\{3\frac{\lambda^n BM}{1-\lambda},3M\bigg(\frac{B}{1-\lambda}\bigg)^2+MZ\}||f||_{{C}^{1}}.
\]
Thus, by Proposition \ref{Th:SwLY}, we get 
\begin{align*}
||\mathcal{L}^nf||_{{C}^{2}}& =||\mathcal{L}^nf||_{{C}^{1}}+||(\mathcal{L}^nf)^{\prime \prime }||_{\infty } \\
&\leq M\lambda^{2n}||f||_{{C}^{2}}\\
&+\bigg(\max\{3\frac{\lambda^n BM}{1-\lambda},3M\bigg(\frac{B}{1-\lambda}\bigg)^2+MZ\}+M\lambda^n(1-\lambda^{n})+M^2\bigg)||f||_{{C}^{1}}.
\end{align*}
\end{proof}
\subsection{Uniform Lasota-Yorke inequality for $\PF_{\eta}$}\label{subsec:uniform_LY_C1C0}
In this subsection we prove uniform Lasota-Yorke inequalities for the discretized operator defined
$\PF_{\eta}=\Pi_{\eta} \PF\Pi_{\eta}$ that was defined in Subsection \ref{newproj}.

\begin{proposition}\label{prop:uniform_LY_C1C0}
Let $0<\eta<\eta_0$. Suppose that\footnote{If $\Theta\ge 1 $; i.e., the expansion of $T$ is 
not big enough, we use an iterate $T^k:=G$ of $T$ as in
Proposition \ref{Th:SwLY}, with an expansion factor that guarantees 
the corresponding $\Theta$ to be strictly smaller than 1. 
Note that, since $\mathcal L$ has a spectral gap on $C^i$, $i=1,2$, 
its iterate $\mathcal L_G$ will also have a spectral gap on $C^i$, $i=1,2$. 
In particular, 1 will still be a simple eigenvalue of $\mathcal L_G$ on  $C^i$, $i=1,2$.}
\[
\Theta:=\bigg(\frac{3}{2}+6\eta_0\bigg)(M+B)\lambda <1.
\]
For any $f\in C^1([0,1])$, and any $n\ge 1$, we have
\[
||\PF_{\eta}^n f||_{C^1}\leq (3+12\eta_0)\bigg(\Theta^n||f||_{C^1}+\frac{B(B+1)}{1-\Theta}||f||_1\bigg)+||f||_1,
\]
and
\[||\PF_{\eta}^n||_{C^1}\leq (3+12\eta_0)\bigg(\Theta^n+\frac{B(B+1)}{1-\Theta}+\frac{1}{3}\bigg).\]
\end{proposition}
\begin{proof}
We start by bounding, using the inequality proved in Proposition \ref{prop:C1L1}
\begin{align*}
||(\PF\Pi_\eta f)'||_{\infty}&\leq (M+B)\lambda ||(\Pi_\eta f)'||_{\infty}+B(B+1) ||\Pi_\eta f||_1\\
&\leq \bigg(\frac{3}{2}+6\eta\bigg)(M+B)\lambda ||f'||_{\infty}+B(B+1) ||f||_1.
\end{align*}
Since $||(\PF\Pi_\eta)^n f||_1=||f||_1$, we have that
\[
||((\PF\Pi_\eta)^n f)'||_{\infty}\leq \Theta^n||f'||_{\infty}+\frac{B(B+1)}{1-\Theta}||f||_1.
\]
Therefore,
\[
||({\PF}_{\eta}^n f)'||_{\infty}\leq \bigg(\frac{3}{2}+6\eta\bigg)\bigg(\Theta^n||f'||_{\infty}+\frac{B(B+1)}{1-\Theta}||f||_1\bigg).
\]
Since
\[
||f||_{\infty}\leq ||f||_{1}+\textrm{Var}(f)\leq ||f||_{1}+||f'||_{\infty}
\]
we have 
\[
||\PF_{\eta}^n f||_{C^1}\leq (3+12\eta)\bigg(\Theta^n||f'||_{\infty}+\frac{B(B+1)}{1-\Theta}||f||_1\bigg)+||f||_1.
\]
\end{proof}
\subsection{Approximating the invariant density in the ${C}^1$ norm}\label{subsec:approxLhatstochastic} 
In this subsection we provide a discretization scheme of the transfer operator $\mathcal L$  in order to approximate 
the invariant density of $T$ in the ${C}^1$ norm.

\subsubsection{\bf An approximation of $\mathcal L$ as an operator from $C^2\to C^1$.} \label{subsec:approxC2C1}
Let
\begin{equation*}
\phi(x)=\left\{ 
\begin{array}{cc}
1+10x^3+15x^4+6x^5 & x\in \lbrack -1,0] \\ 
1-10x^3+15x^4-6x^5 & x\in \lbrack 0,1]\\
0&\textrm{otherwise}
\end{array}%
\right.
\end{equation*}
and
\begin{equation*}
\nu(x)=\left\{ 
\begin{array}{cc}
x-6x^3-8x^4-3x^5 & x\in \lbrack -1,0] \\ 
x-6x^3+8x^4-3x^5 & x\in \lbrack 0,1]\\
0&\textrm{otherwise}
\end{array}%
\right..
\end{equation*}
Let $m\in \mathbb{N}$ and $\eta=1/m$. For $i=0,\dots,m$, let $a_{i}=i/m$, $\phi_i(x)=\phi(m\cdot x-i)$, $\nu_i(x)=\nu(m\cdot x-i)/m$.
Let $\delta_{ij}=1$ if $i=j$ and $\delta_{ij}=0$ if $i\neq j$. The following relations hold for all $i$ and $j$:
\begin{align*}
\phi_i(a_j)&=\delta_{ij},\quad \phi_i'(a_j)=0,\quad \phi_i''(a_j)=0,\\
\nu_i(a_j)&=0,\quad\nu_i'(a_j)=\delta_{ij},\quad \nu_i''(a_j)=0.
\end{align*}
Moreover, $\sum_{i=0}^m \phi_i(x)=1$; i.e., $\{\phi_i\}_{i=0}^m$ forms a partition of unity. Further, for $x\in[a_i,a_{i+1}]$ we have
\[
|\nu_i(x)|+|\nu_{i+1}(x)|\leq \frac{5}{16}\eta,\quad |\nu'_i(x)|+|\nu'_{i+1}(x)|\leq 1.
\] 
Furthermore, $\int_0^1 \nu_i(x)=0$ for all $i=1,\ldots, m-1$. In addition we have,
\begin{align}\label{eq:bounds}
||\phi_i||_{\infty}&=1,\quad ||\phi_i'||_{\infty}=\frac{15 m}{8},\quad ||\phi_i''||_{\infty}=\frac{10\sqrt{3}m^2}{3},\\
||\nu_i||_{\infty}&=\frac{16}{81m},\quad ||\nu_i'||_{\infty}=1,\quad ||\nu_i''||_{\infty}=\frac{8m}{225}(28+19\sqrt{19})\leq 4m.
\end{align}
Let 
\[
\kappa(x)=6x(1-x).
\]
Note that $\int_0^1\kappa(x)dx=1$.
Let
\begin{equation}\label{eq:pol_approx}
p(x)=\sum_{i=0}^m f(a_i)\phi_i(x)+f'(a_i)\nu_i(x).
\end{equation}
We define the operator 
\begin{equation*}
(\tilde{\Pi}_{\eta}f)(x):=p(x)+\bigg(\int_0^1 f dx-\int_0^1 p dx\bigg)\kappa(x).
\end{equation*}
For $f\in C^2([0,1])$, we prove in Lemma \ref{lemma:pi_tilda} estimates on the $C^i$, $i=0,1,2$, norms of  $\tilde\Pi_{\eta}f$. We first start with two preliminary lemmas.

\begin{lemma}\label{lemma:poly}
Let $f\in C^2([0,1])$, and let $p$ be as in \eqref{eq:pol_approx}.
\begin{enumerate}
 \item $||p||_{\infty}\leq ||f||_{\infty}+32/81 ||f'||_{\infty}\eta$,
 \item $||p'||_{\infty}\leq 23/8 ||f'||_{\infty}$,
 \item $||p''||_{\infty}\leq 4||f''||_{\infty}$.
\end{enumerate}
\end{lemma}
\begin{proof}
In this proof, we will denote by $p_i(x):=p(x)|_{[a_i,a_{i+1}]}$.
By construction we have 
\[
p_i(x)=f(a_i)\phi_i(x)+f(a_{i+1})\phi_{i+1}(x)+f'(a_i)\nu_i(x)+f'(a_{i+1})\nu_{i+1}(x).
\]
(1) follows by observing that:
\[
|p_i(x)|\leq ||f||_{\infty}|\phi_i(x)+\phi_{i+1}(x)|+||f'||_{\infty}(|\nu_i(x)|+|\nu_{i+1}(x)|).
\]
The proof of (2) relies on the fact that $(\phi_i(x)+\phi_{i+1}(x))'=0$, since the
$\phi_i$'s form a partition of the unity:
\begin{align*}
|p_i'(x)|\leq |f(a_{i+1})-f(a_i)| |\phi'_i(x)|+||f'||_{\infty}\\
\leq \frac{15}{8}\frac{|f(a_{i+1})-f(a_i)|}{\eta}+||f'||_{\infty}.
\end{align*}
We now prove (3). For $x\in[a_i,a_{i+1}]$ let 
\[
l_i(x):=f(a_{i})+\frac{f(a_i+1)-f(a_i)}{\eta}(x-a_i),\quad q_i(x):=p_i(x)-l_i(x);
\]
we have $p_i''(x)=q_i''(x)$, where $q_i(a_i)=q_i(a_{i+1})=0$.
By unicity of polynomial representation, we have
\[
q(x)=q'(a_i)\nu_i(x)+q'(a_{i+1})\nu_{i+1}(x),
\]
with 
\begin{align*}
q'(a_i)=f'(a_i)-\frac{f(a_i+1)-f(a_i)}{\eta},\quad q'(a_{i+1})=f'(a_{i+1})-\frac{f(a_i+1)-f(a_i)}{\eta},
\end{align*}
and 
\[
|q''(x)|\leq \frac{4}{\eta}(|q'(a_i)|+|q'(a_{i+1}|).
\]
In the last inequality we have used the fact that $|\nu''_i(x)|\leq 4/\eta$. Moreover,
\[
|q'(a_i)|=|f'(a_i)-\frac{f(a_i+1)-f(a_i)}{\eta}|=|f'(a_i)-f'(a_i)-f''(\xi)\eta/2|\leq \frac{||f''||_{\infty}\eta}{2},
\]
which proves (3).
\end{proof}
The following lemma provides bounds on the distance, in $C^i$, $i=0,1,2$, between $f$ and $p$. 
\begin{lemma}\label{lemma:difference}
Let $f\in C^2([0,1])$, and let $p$ be as in \eqref{eq:pol_approx}.
\begin{enumerate}
 \item $||f-p||_{\infty}\leq 15 ||f'||\eta/8$,
 \item $||f-p||_{\infty}\leq 3 ||f''||\eta^2/2$,
 \item $||f'-p'||_{\infty}\leq 3||f''||\eta$,
 \item $|\int_0^1 (f-p) dx|\leq 15 ||f'||\eta/8$,
 \item $|\int_0^1 (f-p) dx|\leq 3 ||f''||\eta^2/2$.
\end{enumerate}
\end{lemma}
\begin{proof}
Let $x \in [a_i,a_i+1]$, we have $ p_i(a_i)=f(a_i)$, $p'(a_i)=f'(a_i)$.
We first prove (1). There exist $\xi,\varsigma\in [a_i,a_{i+1}]$ such that:
\[
|f(x)-p(x)|=|(f'(\xi)-p'(\varsigma))(x-a_i)|\leq |||f'||_{\infty}-||p'||_{\infty}|\eta.
\]
Thus, (1) follows from (2) of Lemma \ref{lemma:poly}. We now prove (2). There exist $\xi,\varsigma\in [a_i,a_{i+1}]$ such that:
\[
|f(x)-p(x)|=|(f''(\xi)-p''(\varsigma))\frac{(x-a_i)^2}{2}|\leq \frac{|||f''||_{\infty}-||p''||_{\infty}|}{2}\eta^2.
\]
Thus, (2) follows from (3) of Lemma \ref{lemma:poly}. We now prove (3); there exist $\xi,\varsigma\in [a_i,a_{i+1}]$ such that:
\[
|f'(x)-p'(x)|=|(f''(\xi)-p''(\varsigma))(x-a_i)|\leq |||f''||_{\infty}-||p''||_{\infty}|\eta.
\]
Thus, (3) follows from (3) of Lemma \ref{lemma:poly}.
Items (4) and (5) follow trivially from item (1) and (2).
\end{proof}
We now obtain estimates on the $C^i$, $i=0,1,2$ norms of $\tilde{\Pi_\eta}f$. First we need some notation that will be used in the remaining lemmas of this subsection. Fix $\eta_0>0$ and define the following constants
$$A_1:=\frac{23}{8}+\frac{32}{81}\eta_0;\quad A_2:= 4+\frac{32}{81}\eta+18\eta^2_0;\quad A_3:=\frac{9}{4}+9.$$
\begin{lemma}\label{lemma:pi_tilda}
Let $f\in C^2([0,1])$, then
\begin{enumerate}
 \item $||\tilde{\Pi}_\eta f||_{\infty}\leq ||f||_{\infty}+32/81 ||f'||_{\infty}\eta+45 ||f'||\eta/16$,
 \item $||\tilde{\Pi}_\eta f||_{\infty}\leq ||f||_{\infty}+32/81 ||f'||_{\infty}\eta+9 ||f''||\eta^2/4$,
 \item $||(\tilde{\Pi}_\eta f)'||_{\infty}\leq 23/8 ||f'||_{\infty}+45 ||f'||\eta/4$,
 \item $||(\tilde{\Pi}_\eta f)'||_{\infty}\leq 23/8 ||f'||_{\infty}+9 ||f''||\eta^2$,
 \item $||(\tilde{\Pi}_\eta f)''||_{\infty}\leq 4||f''||_{\infty}+18 ||f''||\eta^2$.
\end{enumerate}
Moreover, for all $\eta<\eta_0$ we have
\[
||\tilde{\Pi}_{\eta}f||_{C^2}\leq A_{2}||f||_{C^2},\quad ||\tilde{\Pi}_{\eta}f||_{C^1}\leq A_{1}||f||_{C^1}+A_{3}\eta^2 ||f||_{C^2}. 
\]
\end{lemma}
\begin{proof}
By definition
\[
\tilde{\Pi}_\eta f=p(x)+\int_0^1 (f-p) dx \cdot \kappa(x).
\]
We have $||\kappa||_{\infty}= 3/2$, $||\kappa'||_{\infty}=6$, $||\kappa''||_{\infty}=12$.
Consequently, all the items  in the lemma follow from Lemma \ref{lemma:poly} and items (4) and (5)
of Lemma \ref{lemma:difference}.

Finally, we have:
\begin{align*}
 ||\tilde{\Pi}_\eta f||_{C^2}&\leq 4||f||_{C^2}+\frac{32}{81}\eta||f'||_{\infty}+18||f''||\eta^2\leq (4+\frac{32}{81}\eta+18\eta^2)||f||_{C^2};\\
 ||\tilde{\Pi}_\eta f||_{C^1}&\leq \bigg(\frac{23}{8}+\frac{32}{81}\eta\bigg)||f||_{C^1}+\bigg(\frac{9}{4}+9\bigg)||f||_{C^2}\eta^2,
\end{align*}
which completes the proof of the lemma.
\end{proof}

\subsection{Uniform Lasota-Yorke inequality for $\tilde{\PF}_{\eta}$}\label{subsec:uniform_LY_C2C1}
We define now the discretized operator
\[
\tilde{\PF}_{\eta}:=\tilde{\Pi}_\eta\PF\tilde{\Pi}_\eta.
\]
In this subsection we prove uniform Lasota-Yorke inequalities for the discretized operator $\tilde{\PF}_{\eta}$.

\begin{proposition}\label{prop:uniform_bound_C1_tilde}
Let $0<\eta<\eta_0$. Suppose that
\[
\Theta:=(M+B)\lambda (\frac{23}{8}+\frac{45}{4}\eta_0) <1.
\]
For any $n\ge 1$ we have
\[
||(\PF\tilde{\Pi}_{\eta})^n||_{C^1}\leq 2\Theta^n+\frac{B(B+1)}{1-\Theta}+1:=\tilde{K}
\]
and
\[
||\tilde{\PF}_{\eta}^n||_{C^1}\leq 2(\frac{23}{8}+\frac{45}{4}\eta_0)\bigg(\Theta+\frac{B(B+1)}{1-\Theta}+\frac{4}{23}\bigg):= \tilde{M}.
\]
\end{proposition}
\begin{proof}
The proof is identical to the proof of Proposition \ref{prop:uniform_LY_C1C0}, using the fact
that $||(\tilde{\Pi}_\eta f)'||_{\infty}\leq (\frac{23}{8}+\frac{45}{4}\eta)||f'||_{\infty}$, which was proved in Lemma \ref{lemma:pi_tilda}.
\end{proof}

\begin{lemma}\label{prop:uniform_LY_C2C1}
Let $0<\eta<\eta_0$. Suppose that
\[
\tilde{\Theta}:= A_{2} \cdot M\lambda^2+A_{3} \eta_0^2 <1.
\]
For $f\in C^2([0,1])$ and any $n\ge 1$ we have
\[
||\tilde{\PF}_{\eta}^n f||_{C^2}\leq A_{2}\tilde{\Theta}^n ||f||_{C^2}+\frac{A_{2}D A_{1}\tilde{M}}{1-\tilde{\Theta}}||f||_{C^1},
\]
where $\tilde{M}:=2(\frac{23}{8}+\frac{45}{4}\eta_0)\bigg(\Theta+\frac{B(B+1)}{1-\Theta}+\frac{4}{23}\bigg)$.
\end{lemma}
\begin{proof}
We bound
\begin{align*}
||\PF\tilde\Pi_\eta f||_{C^2}&\leq M\lambda^2 ||\tilde\Pi_\eta f||_{C^2}+D||\tilde\Pi_\eta f||_{C^1}\\
&\leq (A_{2} \cdot M\lambda^2+A_{3} \eta^2_0) ||f||_{C^2}+D\cdot A_{1}||f||_{C^1}.
\end{align*}
Then
\[
||(\PF\tilde{\Pi}_\eta)^n f||_{C^2}\leq \tilde{\Theta}^n ||f||_{C^2}+\frac{D\cdot A_{1}\tilde{M}}{1-\tilde{\Theta}}||f||_{C^1},
\]
and
\[
||\tilde{\PF}^n f||_{C^2}\leq A_{2}\tilde{\Theta}^n ||f||_{C^2}+\frac{A_{2}D A_{1}\tilde{M}}{1-\tilde{\Theta}}||f||_{C^1}.
\]
\end{proof}

\subsection{Some approximation inequalities}
In this section we show how to control the error made in iterating the discretized transfer 
operator instead of the transfer operator, under the assumption that the dynamics
satisfies a Lasota-Yorke inequality.

\begin{lemma}\label{18}
Suppose there are two norms $||~||_{s}\geq ||~||_{w}$, such
that $\forall f\in B_s,\forall n\geq 1$%
\begin{equation}
||\mathcal{L}^{n}f||_{s}\leq A\lambda _{1}^{n}||f||_{s}+B||f||_{w}.
\end{equation}%
Let  $\pi _{\delta }$ be a finite rank operator satisfying:
\begin{itemize}
\item $\mathcal{L}_{\delta }=\pi _{\delta }\mathcal{L}\pi _{\delta }$ with $%
||\pi _{\delta }f-f||_{w}\leq K \delta ||f||_{s};$

\item $\pi _{\delta }$, $\mathcal{L}^{i}$ and $\mathcal{L}_{\delta }^{i}$
are bounded for the norm $||~||_{w}:$ $||\pi _{\delta }||_{w}\leq P$ and $%
\forall i>0$, $||\mathcal{L}^{i}||_{w}\leq M$, $||\mathcal{L}_{\delta}^{i}|_{V_0}||_{w}\leq C_i$, $i=0,\ldots,N$. 
\end{itemize}
Then
\begin{eqnarray*}
||(\mathcal{L}^n-\mathcal{L}^n_{\delta })f||_{w} &\leq &K\delta\sum_{k=1}^{n}A\lambda
_{1}^{k-1}C_{n-k}(A\lambda _{1}+PM)||f||_{s}\\&+&K\delta\sum_{k=1}^{n}C_{n-k}(A\lambda _{1}+PM+M)B||f||_{w}.
\end{eqnarray*}
\end{lemma}
\begin{proof}
We have 
\begin{equation*}
||(\PF-\mathcal L_{\delta })f||_{w}\leq ||\pi _{\delta }\PF\pi _{\delta }f-\pi%
_{\delta }\PF f||_{w}+||\pi_{\delta }\PF f- \PF f||_{w}.
\end{equation*}%
Since
\begin{equation*}
\pi_{\delta }\PF\pi _{\delta }f-\pi_{\delta }\PF f=%
\pi_{\delta }\PF(\pi _{\delta }f-f),
\end{equation*}%
and $||\pi _{\delta }f-f||_{w}\leq K\delta ||f||_{s}$, we have
\begin{equation*}
||\pi_{\delta }\PF (\pi _{\delta }f-f)||_{w}\leq PM||\pi _{\delta
}f-f||_{w}\leq PMK\delta ||f||_{s}.
\end{equation*}
On the other hand
\begin{equation*}
||\pi_{\delta }\PF f-\PF f||_{w}\leq K\delta ||\PF f||_{s}\leq K\delta
(A\lambda _{1}||f||_{s}+B||f||_{w})
\end{equation*}
which gives 
\begin{equation}
||(\PF -\PF_{\delta })f||_{w}\leq K\delta (A\lambda _{1}+PM)||f||_{s}+K\delta
B||f||_{w}
\end{equation}

Now let us consider $( \PF_{\delta }^{n}-\PF^{n})f$. We have
\begin{eqnarray*}
||(\PF_{\delta }^{n}-\PF^{n})f||_{w} &\leq& \sum_{k=1}^{n}|| \PF_{\delta
}^{n-k}(\PF_{\delta }-\PF)\PF^{k-1}f||_{w}\leq \sum_{k=1}^{n}C_{n-k}||( \PF_{\delta
}-\PF) \PF^{k-1}f||_{w} \\
&\leq& K\delta \sum_{k=1}^{n}C_{n-k}(A\lambda
_{1}+PM)||\PF^{k-1}f||_{s}+C_{n-k}B||\PF^{k-1}f||_{w} \\
&\leq& K\delta \sum_{k=1}^{n}C_{n-k}\bigg((A\lambda _{1}+PM)(A\lambda
_{1}^{k-1}||f||_{s}+B||f||_{w})+BM||f||_{w}\bigg).
\end{eqnarray*}
We will now collect the terms in front of $||f||_s$:
\[
K\delta \sum_{k=1}^{n}A\lambda
_{1}^{k-1}C_{n-k}(A\lambda _{1}+PM)||f||_{s},
\]
and the terms in front of $||f||_w$:
\[
K\delta\sum_{k=1}^{n}C_{n-k}(A\lambda _{1}+PM+M)B||f||_{w}.
\]
\end{proof}

In the case where $f$ is a fixed point of $\mathcal L$ we have the following estimate:

\begin{lemma}\label{lemma:distance_on_fixed_point}
Suppose there are two norms $||~||_{s}\geq ||~||_{w}$, such
that $\forall f\in B_s,\forall n\geq 1$%
\begin{equation}
||\mathcal{\mathcal L}^{n}f||_{s}\leq A\lambda _{1}^{n}||f||_{s}+B||f||_{w}.
\end{equation}
Let  $\pi _{\delta }$ be a finite rank operator satisfying:
\begin{itemize}
\item $\mathcal{\mathcal L}_{\delta }=\pi _{\delta }\mathcal{\mathcal L}\pi _{\delta }$ with $%
||\pi _{\delta }f-f||_{w}\leq K \delta ||f||_{s}$
\item $\pi _{\delta }$, $\mathcal{\mathcal L}^{i}$ and $\mathcal{\mathcal L}_{\delta }^{i}$
are bounded for the norm $||~||_{w}:$ $||\pi _{\delta }||_{w}\leq P$ and $%
\forall i>0$, $||\mathcal{\mathcal L}^{i}||_{w}\leq M$.
\end{itemize}
Then if $f$ is a fixed point of $\mathcal{\mathcal L}$, we have
\[||\mathcal Lf-\mathcal L_{\delta}f||\leq K\delta (1+PM)||f||_s.\]
\end{lemma}
\begin{proof}
The proof is almost identical to the one above:
\[||\mathcal Lf-\mathcal L_{\delta}f||_w\leq ||\mathcal Lf-\pi_{\delta}\mathcal Lf||_w+||\pi_{\delta}\mathcal Lf-\pi_{\delta}\mathcal L\pi_{\delta} f||_w,\]
since $f$ is fixed point:
\begin{align*}
||\mathcal Lf-\mathcal L_{\delta}f||_w &\leq ||f-\pi_{\delta}f||_w+||\pi_{\delta}\mathcal Lf-\pi_{\delta}\mathcal L\pi_{\delta} f||_w\\ 
&\leq K\delta ||f||_s+P||\mathcal Lf-\mathcal L\pi_{\delta} f||_w\\
&\leq K\delta ||f||_s+PM||f-\pi_{\delta}f||_w\\
&\leq K\delta ||f||_s+PMK\delta ||f||_s.
\end{align*}
\end{proof}

 \subsection{ A recursive convergence to equilibrium estimation for maps
satisfying a Lasota-Yorke inequality\label{sec:GNS}}

Here we recall an algorithm introduced in \cite{GNS} to compute the convergence to equilibrium of a measure preserving system satisfying a Lasota-Yorke inequality. We will see how, the Lasota-Yorke inequality together with a
suitable approximation of the transfer operator by a finite dimensional operator can be used to deduce finite time and asymptotic upper bounds on the contraction of the zero average space.

\bigskip

Consider two vector subspaces of the space of signed measures $B_{s}\subseteq B_{w}$ with norms $||~||_{s}\geq ||~||_{w}$,
suppose that the transfer operator $\mathcal{L}$ is
such that $\forall f\in B_s,\forall n\geq 1$%
\begin{equation}\label{1}
||\mathcal{L}^{n}f||_{s}\leq A\lambda _{1}^{n}||f||_{s}+B||f||_{w}.
\end{equation}%

Let  $\pi _{\delta }$ be a finite rank operator satisfying:
\begin{itemize}
\item $\mathcal{L}_{\delta }=\pi _{\delta }\mathcal{L}\pi _{\delta }$ with $%
||\pi _{\delta }f-f||_{w}\leq K \delta ||f||_{s};$
\item $\pi _{\delta }$, $\mathcal{L}^{i}$ and $\mathcal{L}_{\delta }^{i}$
are bounded for the norm $||~||_{w}:$ $||\pi _{\delta }||_{w}\leq P$ and $%
\forall i>0$, $||\mathcal{L}^{i}||_{w}\leq M$, $||\mathcal{L}_{\delta}^{i}|_{V_0}||_{w}\leq \tilde{C}_i$, $i=0,\ldots,N$. 
\end{itemize}

Then by Lemma \ref{18} there exist $C(\delta,n), D(\delta,n)$ depending
only on $\delta$ and $n$, such that
\begin{equation}
||(\mathcal{L}_{\delta }^{n}-\mathcal{L}^{n})g||_{w}\leq 
C(\delta,n)||g||_{s}+D(\delta,n)||g||_{w}.  \label{2}
\end{equation}

Suppose now that there exists an $n_1$ such that $||\mathcal{L}_{\delta}^{n_1}|_{V_0}||_{w}\leq \tilde{C}_{n_1}<1$;
from now on, we will denote $\lambda_2=\tilde{C}_{n_1}< 1$.

Let us consider a starting measure: $g_{0}\in V_0$, let us denote $%
g_{i+1}=\mathcal L^{n_{1}}g_{i}.$ If the system is as above, putting together the Lasota-Yorke inequality%
\eqref{1} and the approximation inequality \eqref{2}
\begin{equation}\label{2'}
\left\{ 
\begin{array}{c}
||\mathcal{L}^{n_{1}}g_{i}||_{s}\leq A\lambda
_{1}^{n_{1}}||g_{i}||_{s}+B||g_{i}||_{w} \\ 
||\mathcal{L}^{n_{1}}g_{i}||_{w}\leq \lambda _{2}||g_{i}||_{w}+C(\delta,n_1)||g_{i}||_{s}+D(\delta,n_1)||g_{i}||_{w}%
\end{array}%
\right. .
\end{equation}%
Writing \eqref{2'}  in a vector notation:%
\begin{equation}
\left( 
\begin{array}{c}
||g_{i+1}||_{s} \\ 
||g_{i+1}||_{w}%
\end{array}%
\right) \preceq \left( 
\begin{array}{cc}
A\lambda _{1}^{n_{1}} & B \\ 
C(\delta,n_1) & D(\delta,n_1)+\lambda _{2}%
\end{array}%
\right) \left( 
\begin{array}{c}
||g_{i}||_{s} \\ 
||g_{i}||_{w}%
\end{array}%
\right)  \label{4}
\end{equation}%
where $\preceq $ indicates the component-wise $\leq $ relation (both
coordinates are less or equal). The relation $\preceq $ can be used because
the matrix is positive. The relation \eqref{4} and the assumptions allow to 
estimate explicitly the contraction rate, by approximating the matrix and its iterations. 
Let 
\[
\mathcal{M}=\left( 
\begin{array}{cc}
A\lambda _{1}^{n_{1}} & B \\ 
C(\delta,n_1) & D(\delta,n_1)+\lambda _{2}%
\end{array}%
\right) .
\]
Consequently, we can bound $||g_{i}||_{s}$ and $%
||g_{i}||_{w} $ by a sequence 
\begin{equation*}
\left( 
\begin{array}{c}
||g_{i}||_{s} \\ 
||g_{i}||_{w}%
\end{array}%
\right) \preceq \mathcal{M}^{i}\left( 
\begin{array}{c}
||g_{0}||_{s} \\ 
||g_{0}||_{w}%
\end{array}%
\right)
\end{equation*}
which can be computed explicitly. This gives an explicit estimate on the speed of convergence for the norms $||\ ||$ and $||\ ||_{w}$ at a given time.

We need an asymptotic estimation as the one given in (\ref{unicontr}) and
in particular an estimation for $C_{1}$ and $\rho .$ This can be done using the
eigenvalues and eigenvectors of $\mathcal{M}.$

Indeed, let the leading eigenvalue be denoted by $\rho _{\mathcal{M}}$ and a
left positive eigenvector $(a,b)$, such that $a+b=1$. For each pair of
values by $(a,b)$ such that $a+b=1$. We can define a norm 
\begin{equation*}
||g||_{(a,b)}=a||g||_{s}+b||g||_{w}.
\end{equation*}
We have  
\begin{equation*}
||\mathcal{L}g||_{(a,b)}=a||\mathcal{L}g||_{s}+b||\mathcal{L}g||_{w}\leq
(a,b)\cdot \mathcal{M}\cdot \left( 
\begin{array}{c}
||g||_{s} \\ 
||g||_{w}%
\end{array}%
\right).
\end{equation*}
Then
\begin{equation*}
||\mathcal{L}^{kn_{1}}g||_{(a,b)}\leq \rho _{\mathcal{M}}^{k}||g||_{(a,b)}.
\end{equation*}
By estimating $\rho _{\mathcal{M}}$\ and the vector $(a,b)$ we can
have upper estimates on $C_{1}$ and $\rho.$

\end{document}